\documentclass[english,11pt]{article}

\usepackage[T1]{fontenc}
\usepackage{babel}
\usepackage{amsmath,amsfonts,amsthm}
\usepackage{color}

\DeclareMathSymbol{\leqslant}{\mathalpha}{AMSa}{"36} 
\DeclareMathSymbol{\geqslant}{\mathalpha}{AMSa}{"3E} 
\renewcommand{\leq}{\;\leqslant\;}                   
\renewcommand{\geq}{\;\geqslant\;}                   

\newtheorem{Th}{Theorem}
\newtheorem{Le}[Th]{Lemma}
\newtheorem{Pro}[Th]{Proposition}
\newtheorem{Def}[Th]{Definition}
\newtheorem{Cor}[Th]{Corollary}
\newtheorem{Rq}{Remark}

\newcommand{\cA}{\ensuremath{\mathcal A}}
\newcommand{\cB}{\ensuremath{\mathcal B}}

\newcommand{\cD}{\ensuremath{\mathcal D}}
\newcommand{\cE}{\ensuremath{\mathcal E}}
\newcommand{\cF}{\ensuremath{\mathcal F}}

\newcommand{\cH}{\ensuremath{\mathcal H}}

\newcommand{\cP}{\ensuremath{\mathcal P}}

\newcommand{\cR}{\ensuremath{\mathcal R}}
\newcommand{\cS}{\ensuremath{\mathcal S}}

\newcommand{\cX}{\ensuremath{\mathcal X}}

\newcommand{\bbC}{{\ensuremath{\mathbb C}} }
\newcommand{\bbD}{{\ensuremath{\mathbb D}} }
\newcommand{\bbE}{{\ensuremath{\mathbb E}} }

\newcommand{\bbN}{{\ensuremath{\mathbb N}} }

\newcommand{\bbP}{{\ensuremath{\mathbb P}} }

\newcommand{\bbR}{{\ensuremath{\mathbb R}} }

\newcommand{\Om}{\Omega}
\newcommand{\E}{\bbE}
\newcommand{\C}{\bbC}
\newcommand{\N}{\bbN}
\newcommand{\Ne}{\bbN^{\ast}}
\newcommand{\R}{\bbR}
\newcommand{\bbd}{\mathbf{d}}

\newcommand{\bx}{\bar{x}}
\newcommand{\bd}{{\bf d}}

\newcommand{\tbd}{{\bf\tilde{ d}}}
\newcommand{\te}{{\tilde{ e}}}
\newcommand{\tga}{{\tilde{\gamma}}}
\newcommand{\tN}{\tilde{N}}
\newcommand{\hOm}{\hat{\Om}}
\newcommand{\hcA}{\hat{\cA}}
\newcommand{\hbbP}{\hat{\bbP}}
\newcommand{\hE}{\hat{\E}}
\newcommand{\hw}{\hat{w}}
\newcommand{\crea}{\varepsilon^+}
\newcommand{\anni}{\varepsilon^-}
\newcommand{\sbbD}{\underline{\bbD}}
\newcommand{\1}{\mathbf{1}}
\title{Energy image density property
 and\\  the lent particle method for
Poisson measures}\date{}
\author{Nicolas BOULEAU and Laurent DENIS}
\begin{document}
\maketitle
\date{}

\begin{abstract}
 \hspace{.3cm}We introduce a new approach to absolute continuity of laws of
Poisson functionals. It is based on the {\it energy image density} property for
Dirichlet forms. The associated gradient is a local operator and gives rise to a nice formula called {\it the lent particle
method} which consists in adding a particle and taking it back after some calculation.

\end{abstract}
{\bf AMS 2000 subject classifications:} Primary {60G57, 60H05} ;
secondary {60J45,60G51}
\\
{\bf Keywords:}  {Poisson functionals, Dirichlet Forms, Energy
Image Density, Lévy Processes, Gradient}

\section{Introduction}
The aim of this article is to improve some tools provided by Dirichlet forms
for studying the regularity of Poisson functionals. First, the {\it energy image density property}
 (EID) which guarantees the existence of a density for $\mathbb{R}^d$-valued random variables whose
 carr\'e du champ matrix is almost surely regular. Second, the Lipschitz functional calculus for a
 {\it local gradient} satisfying the chain rule, which yields regularity results for functionals of L\'evy processes.

For a local Dirichlet structure with carr\'e du champ, the energy image density property is always
true for real-valued functions in the domain of the form (Bouleau \cite{bouleau1}, Bouleau-Hirsch
 \cite{bouleau-hirsch2} Chap. I \S7). It has been conjectured in 1986 (Bouleau-Hirsch
 \cite{bouleau-hirsch1} p251) that (EID) were true for any $\mathbb{R}^d$-valued function
 whose components are in the domain of the form for any local Dirichlet structure with carr\'e
  du champ. This has been shown for the Wiener space equipped with the Ornstein-Uhlenbeck form and
  for some other structures by Bouleau-Hirsch (cf. \cite{bouleau-hirsch2} Chap. II \S 5 and Chap. V example 2.2.4) and also for the Poisson space by A. Coquio \cite{coquio} when the intensity measure is the Lebesgue measure on an open set, but this conjecture
  being at present neither refuted nor proved in full generality, it has to be established in every
  particular setting. We will proceed in two steps : first (Part 2) we prove sufficient conditions for
  (EID) based mainly on a study of Shiqi Song \cite{song} using a characterization of Albeverio-R\"ockner
   \cite{albeverio-rockner}, then (Part 4) we show that the Dirichlet structure on the
   Poisson space obtained from a Dirichlet structure on the states space inherits from that one the
   (EID) property.

If we think a local Dirichlet structure with carr\'e du champ $(X,\cX
,\nu,\bbd,\gamma)$ as a description of the Markovian movement of a particle on the space $(X,\cX)$
whose transition semi-group $p_t$ is symmetric with respect to the measure $\nu$ and strongly continuous
on $L^2(\nu)$, the construction of the Poisson measure allows to associate to this structure a structure
on the Poisson space $(\Omega, \mathcal{A},\mathbb{P},\mathbb{D},\Gamma)$ which describes similarly the
movement of a family of independent identical particles whose initial law is the Poisson measure with
intensity $\nu$. This construction is ancient and may be performed in several ways.

The simplest one, from the point of view of Dirichlet forms, is based on products and follows faithfully
 the probabilistic construction (Bouleau \cite{bouleau2}, Denis \cite{denis}, Bouleau \cite{bouleau3} Chap.
 VI \S 3). The cuts that this method introduces are harmless for the functional calculus with the carr\'e
 du champ $\Gamma$, but it does not clearly show what happens for the generator and its domain.

Another way consists in using the transition semi-groups (Martin-L\"of \cite{martin-lof}, Wu \cite{wu},
partially Bichteler-Gravereaux-Jacod \cite{bichteler-gravereaux-jacod}, Surgailis \cite{surgailis}).
It is supposed that there exists a Markov process $x_t$ with values in $X$ whose transition semi-group
$\pi_t$ is a version of $p_t$ (cf. Ma-R\"ockner \cite{ma-rockner1} Chap. IV \S 3), the process starting
at the point $z$ is denoted by $x_t(z)$ and a probability space $(W,\mathcal{W},\Pi)$ is considered where
a family $(x_t(z))_{z\in X}$ of independent processes is realized. For a symmetric function $F$, the new
semi-group $P_t$ is directly defined by
$$(P_tF)(z_1,\ldots,z_n,\ldots)=\int F(x_t(z_1),\ldots,x_t(z_n),\ldots)\,d\Pi$$
Choosing as initial law the Poisson measure with intensity $\nu$ on $(X,\mathcal{X})$, it is possible to
show the symmetry and the strong continuity of $P_t$. This method, based on a deep physical intuition,
often used in the study of infinite systems of particles, needs a careful formalization in order to prevent
any drawback from the fact that the mapping $X\ni z\mapsto x_t(z)$ is not measurable in general due to the
independence. For extensions of this method see \cite{klr}.

In any case, the formulas involving the carr\'e du champ and the gradient require computations and key
 results on  {\it the configuration space } from which the construction may be performed as starting point.
  From this point of view the works are based either on the chaos decomposition (Nualart-Vives
  \cite{nualart-vives}) and provide tools in analogy with the Malliavin calculus on Wiener space,
  but non-local (Picard \cite{picard}, Ishikawa-Kunita \cite{ishikawa-kunita}, Picard \cite{picard2}) or on the expression
  of the generator on a sufficiently rich class and Friedrichs' argument (cf. what may be called the German school in spite of its cosmopolitanism, especially \cite{akr} and \cite{ma-rockner2}).

We will follow a way close to this last one. Several representations of the gradient are possible
(Privault \cite{privault}) and we will propose here a new one with the advantages of both locality
(chain rule) and simplicity on usual functionals. It provides a new method of computing the carr\'e du
champ $\Gamma$ --- the lent particle method --- whose efficiency is displayed on some examples. With respect to the announcement \cite{bouleau4} we have introduced a clearer new notation, the operator $\anni$ being shared from the integration by $N$.  Applications to stochastic differential equations driven by L\'evy processes will be gathered in an other article.\\

{\it \noindent Contents

1. Introduction

2. Energy image density property (EID)

3. Dirichlet structure on the Poisson space related to a Dirichlet structure on the states space

4. (EID) property on the upper space from (EID)
 property on the
bottom space and the domain $\mathbb{D}_{loc}$

5. Two examples.}

\section{The Energy Image Density property (EID)}
In this part we give sufficient conditions for a Dirichlet
structure to fulfill (EID) property. These conditions concern
finite dimensional cases and will be extended to the infinite
dimensional setting of Poisson measures in Part 4.

For each positive integer $d$, we denote by $\cB (\R^d )$ the
Borel $\sigma$-field on $\R^d$ and by $\lambda^d$ the Lebesgue measure
on $(\R^d ,\cB (\R^d ))$ and as usually when no confusion is
possible, we shall  denote it  by $dx$.
For $f$ measurable $f_*\nu$ denotes the image of the measure $\nu$ by $f$.

For a $\sigma$-finite measure $\Lambda$ on some measurable space, a Dirichlet form on $L^2(\Lambda)$ with carr\'e du champ $\gamma$
is said to satisfy (EID) if for any $d$ and for any
$\mathbb{R}^d$-valued function $U$ whose components are in the
domain of the form
$$ U_*[({\det}\gamma[U,U^t])\cdot
\Lambda ]\ll \lambda^d $$ where {\rm det} denotes the determinant.

\subsection{A sufficient condition on $(\R^r ,\cB (\R^r
))$}{\label{FiniteStructure}

Given $r\in\Ne$, for any $\cB (\R^r )$-measurable function
$u:\R^r\rightarrow \R$, all $i\in \{ 1,\cdots ,r\}$ and all
$\bx=(x_1 ,\cdots , x_{i-1}, x_{i+1},\cdots,x_{r})\in \R^{r-1}$,
we consider $u_{\bx }^{(i)}:\R \rightarrow \R$ the function
defined by \[\forall s\in \R \ ,u_{\bx}^{(i)} (s)=u ((\bx,s)_i
),\] where $(\bx ,s)_i
=(x_1,\cdots x_{i-1},s,x_{i+1},\cdots ,x_{r} ).$\\
Conversely if $x=(x_1 ,\cdots ,x_r)$ belongs to $\R^r$ we set $x^i
=(x_1 , \cdots , x_{i-1}, x_{i+1},\cdots,x_{r})$.\\
 Then following
standard notation, for any $\cB (\R )$ measurable function $\rho:
\R \rightarrow \R^+$, we denote by $R(\rho )$ the largest open
set on which $\rho^{-1}$ is locally integrable.\\
Finally, we are given $k:\R^r\rightarrow \R^+$ a Borel function and
$\xi=(\xi_{ij})_{1\leq i,j\leq r}$ an $\R^{r\times r}$-valued and
symmetric
Borel function.

 We make the
following assumptions which
generalize Hamza's condition (cf. Fukushima-Oshima-Takeda \cite{fukushima-oshima-takeda} Chap. 3 \S 3.1 ($3^\circ$), p105):\\

{\underline{Hypotheses (HG)}:
\begin{enumerate} \item For any $i\in\{1,\cdots ,r\}$ and
$\lambda^{r-1}$-almost all $\bx\in\{y\in\R^{r-1}:\ \int_{\R}
k^{(i)}_y (s)\, ds >0\}$, $k_{\bx}^{(i)}=0$, $\lambda^1$-a.e. on
$\R\setminus R(k_{\bx}^{(i)})$.
\item There exists an open set $O\subset \R^r$ such that $\lambda^r (\R^r \setminus O)=0$ and $\xi$
is locally elliptic on $O$ in the sense that for any compact subset
$K$, in $O$, there exists a positive constant $c_K$ such that
\[\forall x\in K,\,\forall c\in\mathbb{R}^r\; \sum_{i,j=1}^r \xi_{ij} (x)c_i c_j \geq
c_K |c|^2 .\]
\end{enumerate}
Following Albeverio-R\"ockner, Theorems 3.2 and 5.3 in \cite{albeverio-rockner} and
also R\"ockner-Wielens Section 4 in \cite{rockner-wielens}, we consider $\bd$ the
set of $\cB (\R^r)$-measurable functions $u$ in $L^2 (kdx)$, such
that for any $i\in\{ 1,\cdots ,r\}$, and $\lambda^{r-1}$-almost all
$\bx\in\R^{r-1}$,  $u^{(i)}_{\bx}$ has an absolute continuous
version $\tilde{u}^{(i)}_{\bx}$ on $R(k_{\bx}^{(i)})$ (defined
$\lambda^1$-a.e.) and such that $ \sum_{i,j}\xi_{ij}\frac{\partial
u}{\partial x_i } \frac{\partial u}{\partial x_j } \in L^1 (kdx),$
where
\[\frac{\partial u}{\partial x_i }
=\frac{d\tilde{u}^{(i)}_{\bx}}{ds}. \] Sometimes, we will simply
denote $\frac{\partial}{\partial x_i}$ by $\partial_i $.\\ And we
consider the following bilinear form on $\bd$:
\[\forall u,v\in\bd,\ e[u,v]=\frac12 \int_{\R^r}\sum_{i,j}
\xi_{ij}(x) \partial_i u(x)\partial_j v (x) k(x)\, dx .\] As usual
we shall simply denote $e[u,u]$ by $e[u]$.
We have
\begin{Pro} {\label{closability}}$(\bd ,e)$ is a local Dirichlet form on $L^2 (k dx)$ which admits a
carr\'e du champ operator $\gamma$ given by \[ \forall u,v\in\bd ,
\ \gamma [u,v]=\sum_{i,j} \xi_{ij} \partial_i u\partial_j v .\]
\end{Pro}
\begin{proof} All is clear excepted the fact that $e$ is a closed
form on $\bd$. To prove it, let us consider a sequence
$(u_n)_{n\in\Ne}$ of elements in $\bd$ which converges to $u$ in
$L^2 (kdx)$ and such that $\lim_{n,m\rightarrow +\infty} e[u_n -u_m
]=0$. Let $W\subset O$, an open subset which satisfies
$\bar{W}\subset O$ and such that $\bar{W}$ is compact.

Let $\bd_W$ be the set of $\cB (\R^r)$-measurable functions $u$ in
$L^2 ({\bf 1}_W \times k\,dx)$, such that for any $i\in\{ 1,\cdots
,r\}$, and $\lambda^{r-1}$-almost all $\bx\in\R^{r-1}$,
$u^{(i)}_{\bx}$ has an absolute continuous version
$\tilde{u}^{(i)}_{\bx}$ on $R(({\bf 1}_W\times k)_{\bx}^{(i)})$
and such that $ \sum_{i,j}\xi_{ij}\frac{\partial u}{\partial x_i }
\frac{\partial u}{\partial x_j } \in L^1 ({\bf 1}_W \times
k\,dx),$ equipped with the bilinear form \[\forall u,v\in\bd_W,
e_W[u,v]=\frac12 \int_{W}\sum_{i} \partial_i u(x)\partial_i v (x)
k(x)\, dx=\frac12 \int_{W}\nabla u(x)\cdot \nabla v (x) k(x)\, dx
.\] One can easily verify, since $W$ is an open set, that for all
$\bx \in\R^{r-1}$
\begin{equation}{\label{open}}
S_{\bx}^i (W)\cap R(k_{\bx}^{(i)}) \subset R(({\bf 1}_W \times
k)_{\bx}^{(i)}),
\end{equation}
where $ S_{\bx}^i (W)$ is the open set $\{s\in\R :\ (\bx ,s)_i \in
W\}$.\\
Then it is clear that the function ${\bf 1}_W \times k$ satisfies
property 1. of { (HG)} and as a consequence of Theorems 3.2 and
5.3 in \cite{albeverio-rockner}, $(\bd_W ,e_W)$ is a Dirichlet
form on $L^2 ({\bf 1}_W \times
 kdx)$. \\
We have for all $n,m\in\N$
\[ e_W (u_n -u_m)=\frac12 \int_W  |\nabla u_n (x)-\nabla u_m(x)|^2\, k(x)dx \leq
\frac{1}{c_{\bar{W}}}e (u_n -u_m),\] as $(\bbd ,e_W)$ is a closed form, we conclude that $u$ belongs to $\bd_W$. \\
Consider now an exhaustive sequence $(W_m )$, of relatively
compact open sets in $O$ such that for all $m\in\N$, $\bar{W}_m
\subset W_{m+1}\subset  O$. We have that for all $m$, $u$ belongs
to $\bd_{W_m}$ hence by Theorems 3.2 and 5.3 in
\cite{albeverio-rockner}, for all $i\in\{ 1,\cdots ,r\}$, and
$\lambda^{r-1}$-almost all $\bx\in\R^{r-1}$,  $u^{(i)}_{\bx}$ has
an absolute continuous version on $\bigcup_{m=1}^{+\infty}R(({\bf
1}_{W_m}\times k)_{\bx}^{(i)})$. Using relation \eqref{open}, we
have
$$S_{\bx}^i (O)\cap R(k_{\bx}^{(i)})=\bigcup_{m=1}^{+\infty}S_{\bx}^i (W_m)\cap R(k_{\bx}^{(i)})
 \subset\bigcup_{m=1}^{+\infty}R(({\bf
1}_{W_m}\times k)_{\bx}^{(i)}) .$$ As $\lambda^{r} (\R^r \setminus
O)=0$, we get that for almost all $\bx \in\R^{r-1}$,
$\bigcup_{m=1}^{+\infty}R(({\bf 1}_{W_m}\times k)_{\bx}^{(i)})=R(
k_{\bx}^{(i)})$ $\lambda^1$-a.e. Moreover, by a diagonal
extraction, we have that a subsequence of $(\nabla u_n ) $
converges $kdx$-a.e. to $\nabla u$, so by Fatou's Lemma, we
conclude that $u\in\bd$ and then $\lim_{n\rightarrow +\infty}e[u_n
-u]=0$, which is the desired result.
\end{proof}

 For any $d\in\Ne$, if $u=(u_1
,\cdots ,u_d )$ belongs to $\bd^d$, we shall denote by $\gamma [u]$
the matrix $(\gamma [u_i ,u_j ])_{1\leq i,j\leq d}$.

\begin{Th}{\label{EID1}} (EID) property : the structure $(\mathbb{R}^r, \mathcal{B}(\mathbb{R}^r), k\,dx, \mathbf{ d}, \gamma)$ satisfies
\[\forall d\in\Ne \ \forall u\in \bd^d\ \ u_*[(\det\gamma[u])\cdot
kdx]\ll \lambda^d. \]
\end{Th}
\begin{proof}  Let us mention that a proof was
given by S. Song in \cite{song} Theorem 16, in the more general
case of {\it{classical Dirichlet forms}}. Following his ideas, we present here a shorter proof.\\
The proof is based on the {\it co-area formula } stated by H.
Federer
in  \cite{federer}, Theorems 3.2.5 and 3.2.12.\\
We first introduce the subset $A\subset \R^r$:
\[ A=\{ x\in\R^r :\ x_i \in R(k_{x^i}^{(i)})\ i=1,\cdots ,r\}.\]
As a consequence of property 1. of (HG), $\int_{A^c} k(x)dx =0.$\\
Let $u=(u_1 ,\cdots ,u_d )\in \bbd^d$. We follow the notation and
definitions
introduced by Bouleau-Hirsch in \cite{bouleau-hirsch2}, Chap. II Section 5.1.\\
Thanks to Theorem 3.2 in \cite{albeverio-rockner} and Stepanoff's
Theorem (see Theorem 3.1.9 in \cite{federer} or Remark 5.1.2 Chap. II in
\cite{bouleau-hirsch2}), it is clear that for almost all $a\in A$,
the {\it approximate derivatives} $\makebox{ap}\frac{\partial
u}{\partial x_i}$ exist for $i=1,\cdots ,r$ and if we set: $Ju
=\left[ \det\left( \left( \sum_{k=1}^r
\partial_k u_i
\partial_k u_j \right)_{1\leq i,j\leq d}\right)\right]^{1/2}$, this  is equal $kdx$ a.e. to the
determinant of the {\it approximate Jacobian matrix } of $u$.
Then, by Theorem 3.1.4 in \cite{federer}, $u$ is {\it
approximately differentiable} at almost all points $a$ in $A$.\\
We denote
by $\cH^{r-d}$ the $(r-d)$-dimensional Hausdorff measure on $\R^r$.\\
 As a consequence of Theorems 3.1.8, 3.1.16 and Lemma 3.1.7 in \cite{federer},
 for all $n\in\Ne$, there exists a map $u^n :\R^r \rightarrow
 \R^d$ of class $\mathcal{C}^1$ such that
 \[ \lambda^r (A\setminus \{ x: u(x)=u^n (x)\}) \leq \frac1n\]
 and
 \[ \forall a\in \{ x: u(x)=u^n (x)\},\ \makebox{ap}\frac{\partial
 u}{\partial x_i} (a)= \makebox{ap}\frac{\partial
 u^n}{\partial x_i} (a),\ i=1,\cdots ,r.\]
 Assume first that $d\leq r$. Let  $B$ be a Borelian set in $\R^d$ such that
 $\lambda^r
(B)=0$ . Thanks to the co-area formula we have
\begin{eqnarray*}
\int_{\R^r  }\1_B (u(x))Ju(x)k(x)\, dx &=& \int_{A}\1_B
(u(x))Ju(x)k(x)\, dx\\&=&\lim_{n\rightarrow +\infty}\int_{A\cap \{
u=u^n \}}\1_B (u(x))Ju(x)k(x)\, dx\\&=&\lim_{n\rightarrow
+\infty}\int_{A\cap \{ u=u^n \}}\1_B (u^n (x))Ju^n (x)k(x)\, dx
\\&&\hspace{-2cm}=\lim_{n\rightarrow +\infty}\int_{\R^r}\left(\int_{(u^n )^{-1}(y)}
\1_{A\cap \{ u=u^n \}} (x)\1_B (u^n (x))k(x)d\cH^{r-d}(x)\right) \, d\lambda^r (y)\\
&&\hspace{-2cm}=\lim_{n\rightarrow +\infty}\int_{\R^r}\1_B
(y)\left(\int_{(u^n )^{-1}(y)}\1_{A\cap \{ u=u^n \}} (x)
k(x)d\cH^{r-d}(x)\right) \, d\lambda^r (y)\\
&&\hspace{-2cm}=0\end{eqnarray*} So that, $u_\ast(Ju\cdot
kdx)\ll\lambda_d$. \\We have
\[ Ju=\left[ \det \left(Du\cdot (Du)^t\right) \right]^{1/2}\makebox{ and } \gamma (u)=
Du\cdot\xi\cdot Du^t ,\] where $Du$ is the $d\times r$ matrix:
$\left( \displaystyle\frac{\partial u_i}{\partial x_k}
\right)_{1\leq i\leq d ; 1\leq k\leq r}$ . \\
As $\xi (x)$ is symmetric and positive definite on $O$ and
$\lambda^r (\R^r \setminus O)=0$, we have
\[ \{ x\in A; \ Ju (x)>0\} = \{ x\in A;\  \det (\gamma (u)(x))>0\} \ a.e.,\]
and this ends the proof in this case.\\
Now, if $d>r$, $\det(\gamma (u))=0$ and the result is trivial.
\end{proof}

\subsection{ The case of a product structure}
We consider a
sequence of functions $\xi^i$ and $k_i$, ${i\in\Ne}$, $k_i$ being non-negative Borel functions such that
$\int_{\R^r} k_i(x)\, dx =1$. We assume that for all $i\in\Ne$, $\xi^i$ and
$k_i$ satisfy hypotheses (HG) so that, we can construct, as for
$k$ in the previous subsection, the Dirichlet form $(\bd_i ,e_i)$ on
$L^2 (\R^r ,k_i dx)$ associated to the carr\'e du champ operator
$\gamma_i$ given by:
 \[ \forall u,v\in\bd_i ,
\ \gamma_i [u,v]=\sum_{k,l} \xi_{kl}^i \partial_k u\partial_l v
.\]
 We now consider the product Dirichlet form $(\tbd ,\te
)=\prod_{i=1}^{+\infty} (\bd_i ,e_i )$ defined on the product space
$\left((\R^r )^{\Ne}, (\cB (R^r ))^{\Ne}\right)$ equipped with the
product probability $\Lambda = \prod_{i=1}^{+\infty}k_i dx$.
We denote by $(X_n )_{n\in\Ne}$  the coordinates
maps on $(\R^r )^{\Ne}$.\\
Let us recall that $U=F(X_1 ,X_2 ,\cdots ,X_n ,\cdots)$ belongs to
$\tbd$ if and only if : \begin{enumerate} \item $U$ belongs to
$L^2 \left((\R^r )^{\Ne}, (\cB (\R^r ))^{\Ne} ,\Lambda \right)$.
\item For all $k\in\Ne$ and $\Lambda$-almost all $(x_1 ,\cdots
,x_{k-1} ,x_{k+1},\cdots )$ in $(\R^r )^{\Ne}$, $F(x_1 ,\cdots
,x_{k-1},\cdot ,x_{k+1},\cdots )$ belongs to $\bd_k$. \item $\te
(U)=\displaystyle\sum_k \int_{(\R^r)^{\Ne}}e_k (F(X_1 (x),\cdots
,X_{k-1}(x),\cdot ,X_{k+1}(x),\cdots ))\, \Lambda (dx)<+\infty.$
\end{enumerate}
Where as usual,  the form $e_k$ acts only on the
$k$-th coordinate.\\
It is also well known that $(\tbd ,\te )$ admits a carré du champ
$\tga$ given by
\[ \tga [U]=\sum_k \gamma_k [F(X_1 ,\cdots ,X_{k-1},\cdot ,X_{k+1},\cdots
)](X_k).\] To prove that  (EID) is satisfied by this structure, we
first prove that it is satisfied for a finite product. So, for all
$n\in\Ne$, we consider $(\tbd_n ,\te_n )=\prod_{i=1}^{n} (\bd_i
,e_i )$ defined on the product space $\left((\R^r )^{n}, (\cB (R^r
))^{n}\right)$ equipped with the product probability $\Lambda_n =
\prod_{i=1}^{n}k_i dx$. By restriction, we keep the same notation
as the one introduced for the infinite product. We know that this
structure admits a carr\'e du champ operator $\tga_n$ given by $
\tga_n =\sum_{i=1}^n \gamma_i.$
\begin{Le}{\label{lemmeEID}} For all $n\in\Ne$, the Dirichlet structure $(\tbd_n ,\te_n )$ satifies (EID):\\
\[\forall d\in\Ne \ \forall U\in (\tbd_n )^d\ \ U_*[({\det}\tga_n[U])\cdot
\Lambda_n ]\ll \lambda^d .\]
\end{Le}
\begin{proof} The proof consists in remarking that
this is nothing but a particular case of Theorem \ref{EID1} on
$\R^{nd}$, $\xi$ being replaced by $\Xi$, the diagonal matrix of the $\xi^i$,  and the density being the
product
density.
\end{proof}

 As a consequence of Chapter V Proposition 2.2.3. in Bouleau-Hirsch
\cite{bouleau-hirsch2}, we have
\begin{Th} {\label{EID2}}The Dirichlet structure $(\tbd ,\te )$ satisfies (EID):\\
\[\forall d\in\Ne \ \forall U\in \tbd^d\ \ U_*[({\det}\tga[U])\cdot
\Lambda ]\ll \lambda^d .\]
\end{Th}
\subsection{ The case of structures obtained by injective images}
The following result could be extended to more general images (see Bouleau-Hirsch \cite{bouleau-hirsch2} Chapter V \S 1.3 p 196 {\it et seq.}). We give the statement in the most useful form  for Poisson measures and processes with independent increments.

Let $(\mathbb{R}^p\backslash\{0\},\mathcal{B}(\mathbb{R}^p\backslash\{0\}), \nu, \mathbf{d}, \gamma)$ be a Dirichlet structure on $\mathbb{R}^p\backslash\{0\}$ satisfying (EID). Thus $\nu$ is $\sigma$-finite, $\gamma$ is the carr\'e du champ operator and the Dirichlet form is $e[u]=1/2\int \gamma[u] d\nu$.

Let $U: \mathbb{R}^p\backslash\{0\}\mapsto\mathbb{R}^q\backslash\{0\}$ be an injective map such that $U\in \mathbf{d}^q$. Then $U_*\nu$ is $\sigma$-finite. If we put
$$\begin{array}{rl}
\mathbf{d}_U&=\{\varphi\in L^2(U_*\nu):\varphi\circ U\in\mathbf{d}\}\\
e_U[\varphi]&=e[\varphi\circ U]\\
\gamma_U[\varphi]&=\frac{d\;U_*(\gamma[\varphi\circ U].\nu)}{d\;U_*\nu}
\end{array}
$$ we have
\begin{Pro}
The term $(\mathbb{R}^q\backslash\{0\},\mathcal{B}(\mathbb{R}^q\backslash\{0\}), U_*\nu, \mathbf{d}_U, \gamma_U)$ is a Dirichlet structure satisfying (EID).
\end{Pro}
\begin{proof} a) That $(\mathbb{R}^q\backslash\{0\},\mathcal{B}(\mathbb{R}^q\backslash\{0\}), U_*\nu, \mathbf{d}_U, \gamma_U)$ be a Dirichlet structure is general and does not use the injectivity of $U$ (cf. the case $\nu$ finite in Bouleau-Hirsch \cite{bouleau-hirsch2} Chap. V \S 1 p. 186 {\it et seq.}).\\
b) By the injectivity of $U$, we see that for $\varphi\in \mathbf{d}_U$
$$(\gamma_U[\varphi])\circ U=\gamma[\varphi\circ U]\quad \nu{\mbox{-a.s.}}$$
so that if $f\in(\mathbf{d}_U)^r$
$$f_*[\det\gamma_U[f]\cdot U_*\nu]=(f\circ U)_*[\det\gamma[f\circ U]\cdot\nu]$$
which proves (EID) for the image structure.
\end{proof}
\begin{Rq} Applying this result yields examples of Dirichlet structures
on $\mathbb{R}^n$ satisfying (EID) whose measures are carried by a
(Lipschitzian) curve in $\mathbb{R}^n$ or, under some hypotheses,
a countable union of such curves, and therefore whithout
density.\end{Rq}
\section{Dirichlet structure on the Poisson space related to a
Dirichlet structure on the states space} Let $(X,\cX
,\nu,\bbd,\gamma)$ be a local symmetric Dirichlet structure which
admits a carr\'e du champ operator i.e. $(X,\cX ,\nu )$ is a
measured space called {\it the bottom space}, $\nu$ is $\sigma$-finite
and the bilinear form
\[ e [f,g]=\frac12\int\gamma [f,g]\, d\nu,\]
is a local  Dirichlet form with domain $\bbd\subset L^2 (\nu )$
and carr\'e du champ operator $\gamma$ (see Bouleau-Hirsch
\cite{bouleau-hirsch2}, Chap. I).
We assume that for all $x\in X$, $\{ x\}$ belongs to $\cX$ and
that $\nu$ is diffuse ($\nu(\{x\})=0\;\forall x$).
The generator associated to this Dirichlet structure is
denoted by $a$, its  domain is $\cD (a)\subset \bbd$ and it
generates
the Markovian strongly continuous semigroup $(p_t )_{t\geq 0}$ on $L^2 (\nu )$.

Our aim is to study, thanks to Dirichlet forms methods, functionals of a   Poisson measure $N$,
associated to  $(X,\cX ,\nu )$. It is defined on  the probability space $(\Om,
\cA, \bbP )$ where $\Om$ is the configuration space, the set of
measures which are countable sum of Dirac measures on $X$, $\cA$
is the sigma-field generated by $N$ and   $\bbP$ is the law of $N$
(see Neveu \cite{neveu}). The probability space $(\Om, \cA ,\bbP)$
is called the {\it upper space}.\\
\subsection{Density lemmas}
Let $(F,\cF ,\mu )$ be a probability space such that for all $x\in
F$, $\{ x\}$ belongs to $\cF$ and $\mu$ is diffuse. Let $n\in\Ne$, we denote by $x_1
,x_2 ,\cdots ,x_n$ the coordinates maps on $(F^n ,\cF^{\otimes n},
\mu^{\times n})$ and we consider the random measure
$m=\sum_{i=1}^n \varepsilon_{x_i}$.
\begin{Le} Let $\cS$ be the symmetric sub-sigma-field in
$\cF^{\otimes n} $ and $p\in [1 ,+\infty[$. Sets\break $\{ m(g_1 ) \cdots
 m(g_n) :\ g_i \in L^{\infty} (\mu)\ \forall i=1 ,\cdots , n\}$ and
$\{ e^{m(g)}:\ g \in L^{\infty} (\mu)\}$ are both total in $L^p
(F^n ,\cS ,\mu^{\times n})$ and the set $\{ e^{im(g)}:\ g \in
L^{\infty} (\mu)\}$ is total in $L^p (F^n ,\cS ,\mu^{\times n};
\bbC )$.
\end{Le}
\begin{proof} Because $\mu$ is diffuse, the set $\{ g_1 (x_1 )\cdots g_n (x_n ):\ g_i \in
L^{\infty} (\mu),\,g_i$ with disjoint supports $ \forall i=1 ,\cdots , n\}$ is total in $L^p
(\mu^{\times n})$. Let $G (x_1 ,\cdots ,x_n )$ be a linear
combination of such functions. If $F(x_1 ,\cdots ,x_n )$ is
symmetric and belongs to $L^p (\mu^{\times n})$ then  the
distance in $L^p (\mu^{\times n})$ between $F(x_1 ,\cdots , x_n
)$ and $G(x_{\sigma (1)},\cdots ,x_{\sigma(n)})$ for $\sigma
\in \mathbf{S}$  the set of permutations
on $\{ 1 ,\cdots ,n\},$ does not depend on $\sigma$ and as a consequence is larger than the distance between $F(x_1 ,\cdots , x_n )$ and the barycenter
$\frac{1}{n!}\sum_{\sigma \in\mathbf S}G(x_{\sigma (1)},\cdots
,x_{\sigma(n)})$. So, the set $\{ \frac{1}{n!}\sum_{\sigma
\in\mathbf S}G(x_{\sigma (1)},\cdots ,x_{\sigma(n)}:\ g_i \in
L^{\infty} (\mu)\,,g_i$ with disjoint supports $ \forall i=1 ,\cdots , n\}$ is total in $L^p (F^n
,\cS ,\mu^{\times n})$. We conclude by using the following
property : if $f_i\ i=1 ,\cdots ,n$,  are $\cF$-measurable functions
with disjoint supports then:
$ m(f_1)\cdots m(f_n)=\sum_{\sigma\in\mathbf{S}} f_1
(x_{\sigma(1)}) \cdots f_n (x_{\sigma (n)}).$
\end{proof}
\begin{Le} Let $N_1$ be a random Poisson measure on $(F,\cF
,\mu_1 )$ where $\mu_1$, the intensity of $N_1$, is a finite and diffuse measure, defined on some
probability space $(\Om_1 ,\cA_1 ,\bbP_1 )$ where $\cA_1 =\sigma
(N_1)$. Then, for any $p\in [1 ,+\infty[$, the set $\{ e^{-N_1
(f)}:\ f\geq 0 , f\in L^{\infty} (\mu_1 )\}$ is total in $L^p
(\Om_1 ,\cA_1 ,\bbP_1 )$ and $\{e^{iN_1 (f)}:\  f\in L^{\infty}
(\mu_1 )\}$ is total in $L^p (\Om_1 ,\cA_1 ,\bbP_1 ;\bbC)$.
\end{Le}
\begin{proof} Let us put $P=N_1 (F)$, it is an integer valued
random variable. As $\{ e^{i\lambda P}:\ \lambda\in\R\} $ is total
in $L^p (\N,\cP (\N ),\bbP_P )$ where $\bbP_P$ is the law of $P$,
for any $n\in\Ne$ and any $g\in L^{\infty }(\mu_1 )$, one can
approximate in  $L^p (\Om_1 ,\cA_1 ,\bbP_1 ;\bbC)$ the random
variable
 ${\bf 1}_{\{P=n\}}e^{iN_1 (g)}$ by a sequence of variables of the
 form $\sum_{k=1}^K a_k e^{i\lambda_k P}e^{iN_1 (g)}$ with $a_k
 ,\lambda_k \in\R$, $k=1\cdots K$. But, as a consequence of the
 previous lemma, we know that $\{{\bf 1}_{\{P=n\}}e^{iN_1 (f)}:\
 f\in L^{\infty}(\mu_1 )\}$ is total in $L^p
 (\{P=n\},\cA_1|_{\{P=n\}}, \bbP_1|_{\{P=n\}}; \bbC )$, which
 provides the result.\end{proof}
 We now give the main lemma, with the notation introduced at the
 beginning of this section.
 \begin{Le}{\label{lemmeDensite}} For $p\in[1,\infty[$, the set $\{ e^{-N(f)}:\ f\geq 0 , f\in L^1 (\nu )\cap
 L^{\infty} (\nu)\} $ is total in  $L^p (\Om ,\cA ,\bbP )$ and $\{
 e^{iN(f)}:\ f  \in L^1 (\nu )\cap L^{\infty} (\nu)\} $
 is total in $L^p (\Om ,\cA ,\bbP ; \bbC )$.
 \end{Le}
 \begin{proof} Assume that $\nu $ is non finite. Let $(F_k
 )_{k\in\N}$ be a partition of $\Om$ such that for all $k$, $\nu
 (F_k )$ be finite. By restriction of $N$ to each set $F_k$, we
 construct a sequence of independent Poisson measures $(N_k )$ such
 that $N=\sum_k N_k$. As any variable in $L^p$ is the limit of
 variables which depend only on a finite number of $N_k$, we conclude
 thanks to the previous lemma.
 \end{proof}
 \subsection{Construction using the Friedrichs' argument}
\subsubsection{3.2.1. Basic formulas and pre-generator}
We set $\tN =N-\nu$ then the identity
$ \E[(\tN (f))^2 ] =\int f^2 \, d\nu,$
for $f\in L^1 (\nu )\cap L^2 (\nu)$ can be extended uniquely to
$f\in L^2 (\nu )$ and this permits to define $\tN (f)$ for $f\in L^2
(\nu )$. The Laplace characteristic functional \begin{equation}\label{211}
\E [e^{i\tN (f)}]=e^{-\int (1-e^{i f}+if)\,d\nu}\ \ \ f\in L^2
(\nu)\end{equation}  yields:
\begin{Pro}{\label{pro4}} For all $f\in\bbd$ and all $h\in\cD (a)$,
\begin{equation}\label{212}
\E \left[ e^{i\tN (f)}\left( \tN (a[h])+\frac{i}2 N(\gamma
[f,h])\right)\right]=0.
\end{equation}\end{Pro}\begin{proof}
Derivating  in $0$ the map $t\rightarrow \E \left[ e^{i\tN
(f+ta[h])}\right] $, we have thanks to (\ref{211}),
\begin{equation}\label{213}
\E [e^{i\tN (f)+\int(1-e^{if}+if)d\nu}\tN (a[h])]=\int
(e^{if}-1)a[h]\, d\nu,
\end{equation}
then using the fact that  function $x\mapsto e^{ix}-1$ is Lipschitz
and vanishes in $0$ and the functional calculus related to a local
Dirichlet form (see Bouleau-Hirsch \cite{bouleau-hirsch2} Section I.6) we get
that the member on the right hand side in (\ref{213}) is equal to
\[ -\frac12 \int \gamma [e^{if}-1,h]\, d\nu =-\frac{i}2 \int
e^{if}\gamma[f,h] \, d\nu.\] We conclude by applying once more
(\ref{213}) with $\gamma [f,h]$ instead of $a[h]$.
\end{proof}

 The linear combinations of variables of the form $e^{i\tN (f)}$ with
$f\in \cD (a)\bigcap L^1 (\nu )$ are dense in
$L^2 (\Om
,\cA ,\bbP ;\C )$ thanks to Lemma 8. This is a natural choice for test functions, but, for technical reason, we need in addition that $\gamma[f]$ belongs to $L^2(\nu)$. So we suppose :

\noindent\underline{Bottom core hypothesis} (BC). The bottom structure is such that there exists a subspace $H$ of $\cD (a)\bigcap L^1 (\nu )$ such that $\forall f\in H,\;\gamma[f]\in L^2(\nu)$, and the space  $\cD_0$ of
linear combinations of  $e^{i\tN (f)},\; f\in H,$ is dense in $L^2 (\Om
,\cA ,\bbP ;\C )$.

 This hypothesis will be fulfilled in all cases on $\mathbb{R}^r$ where $ \cD (a)$ contains the $\mathcal{C}^\infty$ functions with compact support and $\gamma$ operates on them.\\
 
  If $U=\sum_p \lambda_p e^{i\tN (f_p)}$ belongs to $\cD_0$, we put
\begin{equation}{\label{214}}
A_0 [U]=\sum_p \lambda_p e^{i\tN (f_p)}(i\tN (a[f_p ])-\frac12 N
(\gamma [f_p])).\end{equation} This is a natural choice as
candidate for the pregenerator of the upper structure, since, as
easily seen using (\ref{214}), it induces the relation
$\Gamma[N(f)]=N(\gamma[f])$ between the carr\'e du champ operators
of the upper and the bottom structures,  which is satisfied in the
case $\nu(X)<\infty$.

One has to note that for the moment,
$A_0$ is not uniquely determined since {\it a priori} $A_0 [U]$
depends on the expression of $U$ which is possibly non unique.
\begin{Pro} \label{pro5}Let $U,V \in\cD_0$, $U=\sum_p \lambda_p e^{i\tN (f_p)}$
and $V=\sum_q \mu_q e^{i\tN (g_q)}$. One has
\begin{equation}{\label{215}}
-\E [A_0 [U]\overline{V}]=\frac12 \E \left[\sum_{p,q}\lambda_p\overline{\mu_q} e^{i\tN
(f_p -g_q )} N(\gamma [f_p ,g_q])\right]
\end{equation}
which is also equal to
\begin{equation}\label{216}
\frac12 \E [\sum_{p,q} F'_p \overline{G'_q} N(\gamma [f_p ,g_q])],
\end{equation}
where $F$ and $G$ are such that $U=F(\tN (f_1 ),\cdots ,\tN (f_n ))$
and $V=G(\tN (g_1 ),\cdots ,\tN (g_m ))$ and $F'_p =\frac{\partial
F}{\partial x_p} (\tN (f_1 ),\cdots ,\tN (f_n )$, $G'_q
=\frac{\partial G}{\partial x_q}(\tN (g_1 ),\cdots ,\tN (g_m ))$.
\end{Pro}
\begin{proof} We have
\[ -\E [A_0 [U]\overline{V}]=-\E \left[ \sum_{p,q}\lambda_p\overline{\mu_q} e^{i\tN (f_p -g_q)}(i\tN (a[f_p ])-\frac12 N
(\gamma [f_p]))\right].\] Thanks to Proposition \ref{pro4},
\begin{eqnarray*}
-\E \left[ \sum_{p,q}\lambda_p\overline{\mu_q} e^{i\tN (f_p -g_q)}i\tN (a[f_p
])\right]=-\frac12\E \left[ \sum_{p,q}\lambda_p\overline{\mu_q} e^{i\tN (f_p -g_q)}N(\gamma
[f_p, f_p -g_q])\right]\\
=\frac12\E \left[ \sum_{p,q}\lambda_p\overline{\mu_q} e^{i\tN (f_p -g_q)}N(\gamma [f_p,
g_q ])\right]-\frac12\E \left[ \sum_{p,q}\lambda_p\overline{\mu_q} e^{i\tN (f_p
-g_q)}N(\gamma [f_p])\right]\end{eqnarray*} which gives the statement.\end{proof}
It remains to prove that $A_0$ is uniquely determined and so is an
operator acting on $\cD_0$. To this end, thanks to the previous
proposition, we just have to prove that the quantity $\sum_{p,q}
F'_p \overline{G'_q} N(\gamma [f_p ,g_q])$ does not depend on the choice of
representations for $U$ and $V$. In the same spirit as Ma-R\"{o}ckner (see \cite{ma-rockner2}), the introduction of a gradient will yield  this non-dependence. Let us mention that the gradient we
introduce is different from the one considered by these authors and is
based on a notion  that we present now.
\subsubsection{3.2.2. Particle-wise product of a Poisson measure and a probability}
We are still considering $N$ the random Poisson measure on $(X,\cX
,\nu)$ and we are given an auxiliary probability space $(R,\cR
,\rho)$. We construct a random Poisson measure $N\odot\rho$ on
$(X\times R ,\cX\otimes \cR ,\nu\times \rho )$ such that if
$N=\sum_i \varepsilon_{x_i}$ then $N\odot\rho =\sum_i
\varepsilon_{(x_i ,r_i )}$ where $(r_i )$ is a sequence of i.i.d.
random variables independent of $N$ whose common law is
$\rho$. Such a random Poisson measure $N\odot\rho$ is sometimes called a {\it marked} Poisson measure. \\
The construction of $N\odot\rho$ follows line by line the one of
$N$. Let us recall it. We first study the case where $\nu$ is
finite and we consider the probability space
\[ (\N , \cP (\N ), P_{\nu (X)})\times (X,\cX ,\frac{\nu}{\nu
(X)})^{\Ne},\] where $P_{\nu (X)}$ denotes the Poisson law with
intensity $\nu (X)$ and we put
\[ N=\sum_{i=1}^Y \varepsilon_{x_i},\qquad(\mbox{with the convention}\; \sum_1^0=0)\]
where $Y, x_1,\cdots ,x_n ,\cdots $ denote the coordinates maps.
We introduce the probability space
\[ (\hOm,\hcA,\hbbP )=(R,\cR ,\rho )^{\Ne},\]
and the coordinates are denoted by $r_1 ,\cdots ,r_n ,\cdots$. On
the probability space $(\N , \cP (\N ), P_{\nu (X)})\times (X,\cX
,\frac{\nu}{\nu (X)})^{\Ne}\times (\hOm,\hcA,\hbbP )$, we define
the random measure $N\odot\rho =\sum_{i=1}^Y \varepsilon_{(x_i
,r_i )}$. It is a Poisson random measure on $X\times R$ with
intensity measure $\nu\times \rho$. For $f\in L^1 (\nu\times \rho)$
\begin{equation}\label{221}\hE [\int_{X\times R} fdN\odot\rho]=\int_X (\int_R f(x,r)
d\rho(r) )N(dx)\ \ \bbP-a.e.\end{equation} and if $f\in L^2
(\nu\times \rho)$ \begin{equation}\label{222}\hE [(\int_{X\times
R} fdN\odot\rho)^2 ]=(\int_X \int_R f d\rho dN)^2 -\int_X (\int_R
fd\rho)^2 dN +\int_X \int_R f^2 d\rho dN,\end{equation} where $\hE$ stands for the expectation
under the probability $\hbbP$.

If $\nu$ is $\sigma$-finite, we extend this construction by a
standard product argument. Eventually in all cases, we have
constructed $N$ on $(\Om, \cA ,\bbP)$ and $N\odot\rho$ on $(\Om,
\cA ,\bbP)\times (\hOm, \hcA ,\hbbP)$, it is a random Poisson
measure on $X\times R$ with intensity measure
$\nu\times\rho$.\\
We now are able to generalize identities (\ref{221}) and
(\ref{222}): \begin{Pro}{\label{pro6}} Let $F$ be an
$\cA\otimes\cX\otimes\cR$ measurable function such that
$\E\int_{X\times R} F^2 \, d\nu d\rho$ and $\E \int_R (\int_X
|F|d\nu)^2 d\rho$ are both finite then the following relation
holds
\begin{equation}\label{223}
\hE [(\int_{X\times R} FdN\odot\rho)^2 ]=(\int_X \int_R F d\rho
dN)^2 -\int_X (\int_R Fd\rho)^2 dN +\int_X \int_R F^2 d\rho
dN,\end{equation}\end{Pro} \begin{proof} Approximating first $F$
by a sequence of elementary functions and then introducing a
partition $(B_k)$ of subsets of $X$ of finite $\nu$-measure, this
identity is seen to be a consequence of (\ref{222}).
\end{proof}
We denote by $\bbP_N$ the measure $\bbP_N =\bbP(dw)N_w (dx)$ on
$(\Om\times X, \cA\otimes \cX)$. Let us remark that $\bbP_N$ and
$\bbP\times\nu$ are singular because $\nu$ is diffuse. \\
We will use the following consequence of the previous proposition :
\begin{Cor}\label{cor7}Let $F$ be an
$\cA\otimes\cX\otimes\cR$ measurable function. If $F$ belongs to
$L^2 (\Om\times X\times R ,\bbP_N\times\rho)$ and $\int
F(w,x,r)\rho (dr)=0$ for $\bbP_N$-almost all $(w,x)$, then  $\int F
dN\odot\rho$ is well-defined and belongs to $L^2 (\bbP\times\hbbP
)$, moreover
\begin{equation}\label{224}
\hE [(\int_{X\times R} FdN\odot\rho)^2 ]=\int F^2 dNd\rho \ \
\bbP\mbox{-}a.e.\end{equation}
\end{Cor}\begin{proof}
If $F$ satisfies hypotheses of Proposition \ref{pro6} then the
result is clear. The general case is obtained by approximation.
\end{proof}
\subsubsection{3.2.3. Gradient and welldefinedness}
From now on, we assume that the Hilbert space $\bbd$ is separable
so that (see Bouleau-Hirsch \cite{bouleau-hirsch2}, ex.5.9 p. 242) the bottom
Dirichlet structure admits a gradient operator in the sense that
there exist a separable Hilbert space $H$ and a continuous linear
map $D$ from $\bbd$ into $L^2 (X,\nu;H)$ such that
\begin{itemize}
\item $\forall u\in \bbd$, $\| D[u ]\|^2_H =\gamma[u]$. \item If
$F:\R\rightarrow \R$ is Lipschitz  then
\[\forall u\in\bbd,\ D[F\circ u]=(F'\circ u )Du.\]
\item If $F$ is $\mathcal{C}^1$ (continuously differentiable) and Lipschitz from $\R^d$ into
$\R$ (with $d\in\N$) then
\[ \forall u=(u_1 ,\cdots ,u_d) \in \bbd^d ,\ D[F\circ
u]=\sum_{i=1}^d (F'_i \circ u ) D[u_i ].\]
\end{itemize}
As only the Hilbertian structure plays a role, we can choose for
$H$
 the  space $L^2 (R,\cR ,\rho)$ where $(R,\cR ,\rho)$ is a
 probability space such that the dimension of the vector space $L^2 (R,\cR
 ,\rho)$ is infinite. As usual, we identify $L^2 (\nu ;H)$ and
 $L^2 (X\times R ,\cX \otimes \cR ,\nu\times\rho)$ and we denote
 the gradient $D$ by  $\flat$:
 \[\forall u\in\bbd,\ Du=u^{\flat} \in L^2 (X\times R ,\cX \otimes \cR
 ,\nu\times\rho).\]
 Without loss of generality, we assume moreover that operator
 $\flat$ takes its values in the orthogonal space of $1$ in $L^2
 (R,\cR ,\rho)$, in other words we take for $H$ the orthogonal of
 $1$. So that we have
 \begin{equation}\label{231}
 \forall u\in\bbd ,\ \int u^{\flat}d\rho =0\ \ \nu\mbox{-}a.e.
 \end{equation}
 Let us emphasize that hypothesis (\ref{231}) although restriction-free, is a key property here (as in many applications to error calculus cf \cite{bouleau3} Chap. V p225 et seq.) Thanks to Corollary \ref{cor7}, it is the feature which will avoid non-local finite difference calculation on the upper space.
Finally,  although not necessary, we assume for simplicity
that constants belong to $\bbd_{loc}$ (see Bouleau-Hirsch
\cite{bouleau-hirsch2} Chap. I Definition 7.1.3.) 
\begin{equation}\label{232}
1\in \bbd_{loc} \makebox{ which implies }\ \gamma [1]=0 \makebox{
and  } 1^{\flat}=0.
\end{equation}
We now introduce the creation and annihilation operators $\crea$ and $\anni$ well-known in quantum
mechanics (see Meyer \cite{meyer},  Nualart-Vives \cite{nualart-vives}, Picard \cite{picard} etc.)  in the following way:
\[\begin{array}{l} \forall x,w\in\Om,\ \crea_x (w)=w{\bf 1}_{\{ x\in supp
\, w\}}+(w+\varepsilon_x) {\bf 1}_{\{ x\notin supp \, w\}}\\
\forall x,w\in\Omega,\ \anni_x(w)=w{\bf 1}_{\{ x\notin supp
\, w\}}+(w-\varepsilon_x) {\bf 1}_{\{ x\in supp \, w\}}.
\end{array}\] 
One can verify that for all $w\in\Om$,
\begin{equation}\label{233}
\crea_x (w)=w\makebox{ and }\anni_x(w)=w-\varepsilon_x \makebox{ for } N_w \makebox{-almost all }x
\end{equation}
and
\begin{equation}\label{234}
\crea_x (w)=w+\varepsilon_x\makebox{ and }\anni_x(w)=w  \makebox{ for } \nu \makebox{-almost
all }x
\end{equation}
We extend this operator to the functionals by setting:
\[ \crea H(w,x)=H(\crea_x w, x)\quad\makebox{ and }\quad\anni H(w,x)=H(\anni_xw,x).\]
The next lemma shows that the image of $\bbP\times\nu$ by
$\crea$ is nothing but $\bbP_N$ whose image by $\anni$ is $\bbP\times\nu$ :
\begin{Le}{\label{lem8}} Let $H$ be $\cA\otimes\cX$-measurable
and non negative, then
\[\E \int \crea H d\nu =\E\int H dN\quad\makebox{ and }\quad \E\int\anni HdN=\E\int Hd\nu.\]
\end{Le}
\begin{proof} Let us assume first that $H=e^{-N(f)}g$ where $f$ and $g$
are non negative and belong to $L^1 (\nu)\cap L^2 (\nu)$. We have:
\[ \E \int \crea Hd\nu =\E \int e^{-N(f)} e^{-f(x)}g(x)d\nu (x),\]
and by standard calculations based on the properties of the
Laplace functional we obtain that
\[\E \int e^{-N(f)} e^{-f(x)}g(x)d\nu (x)=\E[e^{-N(f)}N(g)]=\E
\int H dN.\] We conclude using a monotone class argument and similarly for the second equation.
\end{proof}
Let us also remark that if $F\in L^2 (\bbP_N\times\rho)$
satisfies $\int Fd\rho=0$ $\bbP_N$-a.e. then if we put $\crea F
(w,x,r)=F(\crea_x (w),x,r)$ we have
\begin{equation}\label{235}
\int\crea F dN\odot\rho=\int F dN\odot\rho \ \
\bbP\mbox{-}a.e.\end{equation} Indeed $\int (\crea F -F)^2 dNd\rho =0 $
$\bbP$-a.e. because $\crea_x (w)=w$ for $N_w$-almost all $x$.
\begin{Def}\label{def9} For all $F\in\cD_0$, we put
\[ F^\sharp =\int \anni((\crea F)^{\flat})\, d\,N\!\odot\!\rho.\]
\end{Def}
Thanks to hypothesis (\ref{232}) we  have the following
representation of $F^\sharp$:
\[F^\sharp (w,\hw)=\int_{X\times R} \anni((F(\crea_{\cdot}(w))-F(w))^\flat)
(x,r)\;N\odot\rho (dxdr).\] Let us also remark that Definition
\ref{def9} makes sense because for all $F\in\cD_0$ and
$\bbP$-almost all $w\in\Omega$, the map $y\mapsto
F(\crea_{y}(w))-F(w)$ belongs to $\bbd$. To see this, take
$F=e^{i\tN (f)}$ with $f\in \cD (a)\bigcap L^1 (\nu)$, then
\[ F(\crea_{y}(w))-F(w)=e^{i\tN (f)}(e^{if(y)}-1),\]
and we know that $e^{if}-1\in \bbd$.
We now proceed and obtain
\begin{equation*}
(e^{i\tN (f)})^\sharp= \int \anni(e^{i\tN (f)} (e^{if}-1)^\flat) \;dN\odot\rho= \int \anni(e^{i\tN (f)+if}(if)^\flat) \;dN\odot\rho\\
 \end{equation*}
and eventually
\[(e^{i\tN (f)})^\sharp=\int e^{i\tN (f)}(if)^\flat \;dN\odot\rho .\]
So, if $F,G\in\cD_0$, $F=\sum_p \lambda_p e^{i\tN (f_p )}$,
$G=\sum_q \mu_q e^{i\tN (g_q )}$, as $\int f_p^\flat d\rho =\int
g_q^\flat d\rho =0$ and thanks to Corollary \ref{cor7}, we have
\[\hE [F^\sharp \overline{G^\sharp} ]=\sum_{p,q} \lambda_p \overline{\mu_q} e^{i\tN (f_p
-g_q )}\int (if_p )^\flat \overline{(ig_q )^\flat} dNd\rho,\] and so
\begin{equation}\label{236}\hE [F^\sharp \overline{G^\sharp} ]
=\sum_{p,q} \lambda_p \overline{\mu_q} e^{i\tN (f_p -g_q )}N(\gamma(f_p ,g_q)
)
\end{equation}
But, by Definition \ref{def9}, it is clear that $F^\sharp$ does
not depend on the representation of $F$ in $\cD_0$ so as a
consequence of the previous identity $\sum_{p,q} \lambda_p \overline{\mu_q}
e^{i\tN (f_p -g_q )}N(\gamma(f_p ,g_q ))$ depends only on $F$ and
$G$ and thanks to  (\ref{215}), we conclude that  $A_0$ is well-defined and is a linear operator from $\cD_0$ into $L^2
(\bbP )$.
\subsubsection{3.2.4. Upper structure and first properties}
As a consequence of Proposition \ref{pro5}, it is clear that $A_0$
is symmetric, non positive on $\cD_0$ therefore (see
Bouleau-Hirsch \cite{bouleau-hirsch2} p.4) it is closable and we can consider
its Friedrichs extension $(A,\cD (A))$ which generates a closed
Hermitian form $\cE$ with domain $\bbD\supset\cD (A)$ such that
\[\forall U\in\cD (A)\ \forall V\in\bbD ,\ \cE (U,V)=-\E
[A[U]\overline{V}].\] Moreover, thanks to Proposition \ref{pro5}, it is clear
that contractions operate, so (see Bouleau-Hirsch \cite{bouleau-hirsch2} ex. \!\!\!3.6 p.16)
$(\bbD ,\cE )$ is a local Dirichlet form which admits a carré du
champ operator $\Gamma$. The upper
structure that we have obtained $(\Omega,\cA,\bbP ,\bbD,\Gamma)$
satisfies the following properties :
\begin{itemize}\item$\forall f\in\bbd,\ \tN (f)\in\bbD
\makebox{ and }$ \begin{equation}\label{241}  \Gamma [\tN
(f)]=N(\gamma [f]),
\end{equation}
moreover the map $f\mapsto \tN (f)$ is an isometry from $\bbd$
into $\bbD$. \item $ \forall f\in\cD (a),\ e^{i\tN (f)}\in \cD
(A),\makebox{ and }$\begin{equation}\label{242} A[e^{i\tN (f)}]
=e^{i\tN (f)}(i\tN (a[f])-\frac12N(\gamma[f])).
\end{equation}
\item The operator $\sharp$ (defined on $\cD_0$) admits an
extension on $\bbD$, still denoted $\sharp$, it is a gradient
associated to $\Gamma$ and for all $f\in\bbd$:
\begin{equation}\label{243}(\tN (f))^{\sharp}=\int_{X\times R}
f^\flat\; dN\odot\rho.\end{equation}
\end{itemize}
As  a gradient for the Dirichlet structure $(\Omega,\cA,\bbP
,\bbD,\Gamma)$,  $\sharp$ is a closed operator from $L^2 (\bbP )$
into $L^2 (\bbP \times \hbbP )$. It satisfies the chain rule and
operates on the functionals of the form $\Phi (\tN (f))$, $\Phi$ Lipschitz $f\in\bbd$, or
more generally $\Psi (\tN (f_1 ), \cdots ,\tN (F_n ))$
with $\Psi$ Lipschitz  and $\mathcal{C}^1$ and $f_1$, $\cdots $ , $f_n$ in $\bbd$.\\
Let us also remark that if $F$ belongs to $\cD_0$,
\begin{equation}\label{244}
A[F]=N(\anni(a[\crea F])).\end{equation}
\subsubsection{3.2.5. Link with the Fock space}
The aim of this  subsection is to make the link with other
existing works and to present another approach based on the
Fock space. It is independent of the rest of this article.\\
Let $g\in\cD(a )\cap L^1 (\nu )$ such that $-\frac12 \leq g\leq 0$
and $a[g]\in L^1 (\nu )$. Clearly, $f=-\log (1+g)$ is non-negative
and belongs to $\bbd$. We have for all $v\in \bbd \cap L^1 (\nu )$
\begin{eqnarray*}
\cE [e^{-N(f)},e^{-N (v)}]&=& \frac12 \E \left[e^{-N(f)}e^{-N(v)}
\Gamma [N(f),N(v)]\right]\\
&=& \frac12 \E \left[e^{-N(f)}e^{-N(v)} N(\gamma [f,v])\right]\\
&=& \frac12  e^{\int_X (1-e^{-f-v})d\nu} \int_X \gamma
[f,v]e^{-f-v}d\nu .
\end{eqnarray*}
As a consequence of the functional calculus
\[ \int_X \gamma
[f,v]e^{-f-v}d\nu=\int_X \gamma [g, e^{-v}]d\nu =-2 \int_X
a[g]e^{-v}d\nu,\] this yields
\begin{equation}\label{251}
\cE [e^{-N(f)},e^{-N (v)}]=-\E [e^{-N(f)}e^{-N(v)}
N(\frac{a[g]}{1+g})].\end{equation} Thus by Lemma
\ref{lemmeDensite}, we obtain
\begin{Pro}\label{pro10}  Let $g\in\cD(a )\cap L^1 (\nu )$ such that $-\frac12 \leq g\leq 0$
and $a[g]\in L^1 (\nu )$ then \begin{equation}\label{252}
e^{N(\log (1+g))}\in\cD (A) \makebox{ and  }A[e^{N(\log
(1+g))}]=e^{N(\log (1+g))}N(\frac{a[g]}{1+g}).\end{equation}
\end{Pro}
Let us recall that $(p_t )$ is the semigroup associated to the bottom
structure. If $g$ satisfies the hypotheses of the
previous proposition, $p_t g$ also satisfies them. The map $\Psi:
t\mapsto e^{N(\log (1+p_t g))} $ is differentiable and
$\frac{d\Psi}{dt}=A\Psi$ with $\Psi (0)=e^{N(\log (1+g))}$ hence
$\Psi (t)=P_t [e^{N(\log (1+g))}]$ where $(P_t )$ is the strongly
continuous semigroup generated by $A$. So, we have proved
\begin{Pro}\label{pro11} Let $g$ be a measurable function with
$-\frac12 \leq g\leq 0$, then
\[ \forall t\geq 0 ,\ P_t [e^{N(\log (1+g))}]=e^{N(\log(1+p_t
g))}.\]
\end{Pro}
For any $m\in\Ne$, we denote by $L^2_{sym} (X^m ,\cX^{\otimes m},
\nu^{\times m})$ the set of symmetric functions in $L^2 (X^m
,\cX^{\otimes m}, \nu^{\times m})$ and we recall that $\nu$ is
diffuse.\\
For all $F\in L^2_{sym} (X^m ,\cX^{\otimes m}, \nu^{\times m})$, we
put
\[ I_m (F)=\int_{X^m} F(x_1 ,\cdots ,x_m ){\bf 1}_{\{\forall i\neq j ,
x_i \neq x_j\}}\, \tN (dx_1 )\cdots \tN (dx_m ).\] One can easily
verify that for all $F,G \in L^2_{sym} (X^m ,\cX^{\otimes m},
\nu^{\times m})$ and all $n,m \in\Ne$,\break $\E [I_m (F)I_n (G)]=0$ if
$n\neq m$ and
\[ \E [I_n (F)I_n (G)]=n! \langle F,G\rangle_{L^2_{sym} (X^n ,\cX^{\otimes n},
\nu^{\times n})},\] where $\langle \cdot,\cdot\rangle_{L^2_{sym}
(X^n ,\cX^{\otimes n}, \nu^{\times n})}$ denotes the scalar product
in $L^2_{sym} (X^n ,\cX^{\otimes n}, \nu^{\times n})$. For all
$n\in\Ne$, we consider $C_n$, the Poisson chaos of order $n$, i.e.
the sub-vector space of $L^2 (\Om,\cA, \bbP)$ generated by the
variables $I_n (F)$, $F\in L^2_{sym} (X^n ,\cX^{\otimes n},
\nu^{\times n})$. The fact that
\[L^2 (\Om,\cA, \bbP)=\R\oplus_{n=1}^{+\infty} C_n .\]
has been proved by K. Ito (see \cite{ito}) in 1956. This proof  is based on
the fact that the set $\{ N(E_1)\cdots N(E_k),$ $ (E_i)$ disjoint sets in  $\cX \}$ is total in $L^2 (\Om,\cA, \bbP)$.

Another approach, quite natural, consists in studying carefully, for $g\in L^1\cap L^\infty(\nu)$,
what has to be  subtracted from the integral with respect to the
product measure
\[ \int_{X^n} g(x_1 )\cdots g(x_n)\, \tN (dx_1)\cdots \tN (dx_n)\]
 to obtain the Poisson stochastic integral
\[ I_n (g^{\otimes n})=\int_{X^n} g(x_1 )\cdots g(x_n){\bf 1}_{\{ \forall i\neq j ,x_i \neq x_j \}}\,
\tN (dx_1)\cdots \tN (dx_n).\] This can be done in an elegant way by
the use of lattices of partitions and the M\"{o}bius inversion
formula (see  Rota-Wallstrom \cite{rota-wallstrom}). This leads to the following formula (observe the tilde on the first $N$ only) :
 \[ I_n (g^{\otimes n})=\sum_{k=1}^n
B_{n,k}(\tilde{N}(g), -1!N(g^2 ),2! N(g^3 ),\ldots,
(-1)^{n-k}(n-k)!N(g^{n-k+1})),\] where the $B_{n,k}$ are the exponential Bell polynomials given by
 \[ B_{n,k}=\sum
\displaystyle\frac{n!}{c_1!c_2 !\cdots
(1!)^{c_1}(2!)^{c_2}\cdots}x_1^{c_1}x_2^{c_2}\cdots\] the sum being
taken over all the non-negative integers $c_1,c_2,\cdots$ such that
\begin{eqnarray*}
c_1 +2c_2 +3c_3 +\cdots&=&n\\
c_1 +c_2+\cdots =k. \end{eqnarray*}
 $I_n (g^{\otimes
n})$ is a homogeneous function of order $n$ with respect
to $g$.
If we express the Taylor expansion of $e^{N(\log (1+tg))}$ and
compute the $n$-th derivate with respect to $t$ thanks to the
formula of the composed functions (see Comtet \cite{comtet}) we obtain
\[ e^{N(\log (1+tg))-t\nu(g)}=1+\sum_{n=1}^{+\infty}\frac{t^n}{n!}
\sum_{k=1}^n B_{n,k} (\tilde{N}(g), -1!N(g^2 ),\ldots,
(-1)^{n-k}(n-k)!N(g^{n-k+1}))\] this yields
\begin{equation}\label{252}
e^{N(\log(1+g))-\nu(g)}=1+\sum_{n=1}^{+\infty}\frac{1}{n!}I_n (g^{\otimes
n}).\end{equation} The density of the chaos is now a consequence of
Lemma \ref{lemmeDensite}.

Conversely, one can prove formula (\ref{252}) thanks to the
density of the chaos, see for instance Surgailis \cite{surgailis}.
By transportation of structure, the density of the chaos has a short proof using stochastic calculus for the Poisson process on $\mathbb{R}_+$, cf Dellacherie, Maisonneuve and Meyer \cite{dellacherie} p207, see also
Applebaum \cite{applebaum} Theorems 4.1 and 4.3.

\subsection{Extension of the representation of the gradient and the lent particle
method} \subsubsection{3.3.1. Extension of the representation of the gradient} The goal
of this subsection is to extend formula of
Definition \ref{def9} to any $F\in\bbD$.\\
To this aim, we introduce an auxiliary vector space $\sbbD$ which
is the completion of the algebraic tensor product $\cD_0\otimes
\bbd$ with respect to the  norm $\| \ \|_{\sbbD}$ which is defined
as follows. \\
Considering $\eta$, a fixed strictly positive function on $X$ such that $N(\eta
)$ belongs to $L^2 (\bbP )$, we set for all $H\in \cD_0 \otimes
\bbd$:
$$
\begin{array}{rl} \| H\|_{\sbbD}&=\displaystyle\left( \E\int_X \anni(\gamma
[H])(w,x)N(dx)\right)^{\frac12} +\E\int(\anni|H|)(w,x)\eta
(x)N(dx)\\
&=\displaystyle\left( \E\int_X \gamma
[H])(w,x)\nu(dx)\right)^{\frac12} +\E\int |H|(w,x)\eta
(x)\nu(dx)
\end{array}
$$

One has to note that if $F\in\cD_0$ then $\crea F-F
\in\cD_0 \otimes \bbd$ and if $F=\sum_p \lambda_p e^{i\tN (f_p
)}$, we have
\[\gamma [\crea F-F]=\sum_{p,q}\lambda_p
\overline{\lambda_q} e^{i\tN (f_p -f_q )}e^{i(f_p -f_q)} \gamma [f_p ,f_q
],\] so that 
\[ \int_X \anni\gamma [\crea F-F]\, dN =\int \sum_{p,q}\lambda_p
\overline{\lambda_q} e^{i\tN (f_p -f_q )} \gamma [f_p ,f_q ]\, dN,\] by the
construction of Proposition \ref{pro5}, this last term is nothing
but $\Gamma [F]$. Thus, if $F\in\cD_0$ then $\crea F-F \in\sbbD$
and \begin{eqnarray*} \| \crea F-F\|_{\sbbD}&=& \left( \E \Gamma
[F]\right)^{\frac12} +\E [\int |\crea F-F|\eta \, dN]\\
&\leq& \left(2\cE [F]\right)^{\frac12}+2\| F\|_{L^2 (\bbP )}\|
N(\eta)\|_{L^2 (\bbP )}
\end{eqnarray*}
As a consequence, $\crea-I$ admits a unique  extension on $\bbD$. It is a continuous linear map
from $\bbD$ into $\sbbD$. Since by (13) $\gamma[\crea F-F]=\gamma[\crea F]$ and $(\crea F-F)^\flat=(\crea F)^\flat$, this leads to the following theorem :
\begin{Th}
The formula
\begin{equation}\label{31}
\forall F\in\bbD,\quad\ F^\sharp =\int_{X\times R} \anni((\crea F )^\flat)\,
dN\odot \rho ,\end{equation} is justified by the following
decomposition:
$$F\in\mathbb{D}\quad\stackrel{\crea-I}{\longmapsto}\quad \varepsilon^+F-F\in\underline{\mathbb{D}}\quad\stackrel{\anni((.)^\flat)}{\longmapsto}\quad\anni((\varepsilon^+F)^\flat)\in L^2_0(\mathbb{P}_N\times\rho)\quad\stackrel{d(N\odot\rho)}{\longmapsto}\quad F^\sharp\in L^2(\mathbb{P}\times\hat{\mathbb{P}})$$
where each operator is continuous on the range of the preceding one and where $L^2_0
(\bbP_N \times \rho )$ is the closed set of elements $G$ in $L^2
(\bbP_N \times \rho )$ such that $\int_R G d\rho=0$ $\bbP_N$-a.e.\\
Moreover, we have for all $F\in\bbD$
\[\Gamma [F]=\hE(F^\sharp )^2 =\int_X \anni(\gamma [\crea F ])\, dN.\]
\end{Th}
\begin{proof}
Let $H\in\sbbD$, there exists a sequence $(H_n)$ of elements in
$\cD_0 \otimes \bbd$ which converges to $H$ in $\sbbD$ and we have
for all $n\in\N$
\[\int \anni(H_n^\flat )^2 \, d\bbP_N d\rho =\E\int \anni\gamma [H_n]\, dN\leq \| H_n \|^2_{\sbbD},\] therefore $(H_n^\flat )$
is a Cauchy sequence in $L^2_0 (\bbP_N \times \rho )$ hence
converges to an element in $L^2_0 (\bbP_N \times \rho )$ that we
denote by $\anni(H^\flat)$.\\
Moreover, if $K\in L^2_0 (\bbP_N \times \rho )$, we have
\[ \E\hE \left( \int_{X\times R} K(w,x,r)\, N\odot \rho
(dxdr)\right)^2 =\E \int_{X\times R} K^2 \, dNd\rho= \|
K\|^2_{L^2 (\bbP_N \times \rho )}.\] This provides the assertion of the statement.
\end{proof}

The functional calculus for $\sharp$ and $\Gamma$ involves mutually singular measures and may be followed step by step :

Let us first recall that by Lemma 13 the map $(w,x)\mapsto(\crea_x(w),x)$ applied to classes of functions $\mathbb{P}_N$-a.e. yields classes of functions $\mathbb{P}\times\nu$-a.e. and also the map $(w,x)\mapsto(\anni_x(w),x)$ applied to classes of functions $\mathbb{P}\times\nu$-a.e. yields classes of functions $\mathbb{P}_N$-a.e.

But product functionals of the form $F(w,x)=G(w)g(x)$ where $G$ is a class $\mathbb{P}$-a.e. and $g$ a class $\nu$-a.e. belong necessarily to a single class $\mathbb{P}_N$-a.e. Hence, if we applied $\crea$ to such a functional, this yields a unique class $\mathbb{P}\times\nu$-a.e. In particular with $F=e^{i\tilde{N}f}g$ :
$$\crea(e^{i\tilde{N}f}g)=e^{i\tilde{N}f}e^{if}g\quad\mathbb{P}\times\nu\makebox{-a.e.}$$
from this class the operator $\anni$ yields a class $\mathbb{P}_N$-a.e.
$$\anni(e^{i\tilde{N}f}e^{if}g)=e^{i\tilde{N}f}g\quad\mathbb{P}_N\mbox{-a.e.}$$
and this result is the same as $F$ $\mathbb{P}_N\mbox{-a.e.}$

This applies to the case where $F$ depends only on $w$ and is defined $\mathbb{P}$-a.e. then
$$\anni(\crea F))=F\quad\mathbb{P}_N\mbox{-a.e.}$$
Thus the functional calculus decomposes as follows :

\begin{Pro}\label{pro13} Let us consider the subset of $\mathbb{D}$ of functionals of the form $H=\Phi(F_1,\ldots,F_n)$ with $\Phi\in\mathcal{C}^1\cap Lip(\mathbb{R}^n)$ and $F_i\in\bbD$, putting $F=(F_1,\ldots,F_n)$ we have the following :
$$\begin{array}{rrlll}
a)&
(\crea H)^\flat&=\sum_i\Phi^\prime_i(\crea F)(\crea F_i)^\flat&\hspace{-.5cm}\mathbb{P}\times\nu\times\rho\mbox{-a.e.}\\
&\gamma[\crea H]&=\sum_{ij}\Phi^\prime_i(\crea F)\Phi^\prime_j(\crea F)\gamma[\crea F_i,\crea F_j]&\mathbb{P}\times\nu\mbox{-a.e.}\\
b)&\anni(\crea H)^\flat&=\sum_i\Phi^\prime_i(F)\anni(\crea F_i)^\flat&\hspace{-.5cm}\mathbb{P}_N\times\rho\mbox{-a.e.}\\
&\anni\gamma[\crea H]&=\sum_{ij}\Phi^\prime_i(F)\Phi^\prime_j(F)\anni\gamma[\crea F_i,\crea F_j]&\mathbb{P}_N\mbox{-a.e.}\\
c)&H^\sharp=\int\anni((\crea H)^\flat)\,dN\odot\rho&=\sum_i\Phi_i^\prime(F)\int\anni(\crea F_i)^\flat\;dN\odot\rho&\mathbb{P}\times\hat{\mathbb{P}}\mbox{-a.e.}\\
&\Gamma[H]=\int\anni\gamma[\crea H]dN&=\sum_{ij}\Phi^\prime_i(F)\Phi^\prime_j(F)\int\anni\gamma[\crea F_i,\crea F_j]dN&\mathbb{P}\mbox{-a.e.}
\end{array}$$
\end{Pro}
\begin{Rq} The projection of the measure $\mathbb{P}_N$ on $\Omega$ is a (possibly non $\sigma$-finite) measure equivalent to $\mathbb{P}$ only if $\nu(X)=+\infty$, i.e. if $\mathbb{P}\{N(1)>0\}=1$.

If $\nu(X)=\|\nu\|<+\infty$, then $\mathbb{P}\{N(1)=0\}=e^{-\|\nu\|}>0$, and the sufficient condition for existence of density $\Gamma[F]>0\;\mathbb{P}$-a.s. is never fulfilled because $\Gamma[F]=\int\anni(\gamma[\crea F])\,dN$ vanishes on $\{N(1)=0\}$. Conditioning arguments with respect to the set $\{N(1)>0\}$ have to be used.
\end{Rq}
\subsubsection{3.3.2. The lent particle method: first application}
The preceding theorem provides a new method to study the regularity of Poisson functionals, that we present on an example.

Let us consider, for instance, a real process $Y_t$ with independent increments and L\'evy measure $\sigma$ integrating $x^2$, $Y_t$ being supposed centered without Gaussian part. We assume that $\sigma$ has a density satisfying Hamza's condition (Fukushima-Oshima-Takeda \cite{fukushima-oshima-takeda} p105) so that a local Dirichlet structure may be constructed on $\mathbb{R}\backslash\{0\}$ with carr\'e du champ
$
\gamma[f]=x^2f^{\prime 2}(x).
$
We suppose also hypothesis (BC) (cf \S 3.2.1).
If $N$ is the random Poisson measure with intensity $dt\times\sigma$ we have $\int_0^th(s)\;dY_s=\int{\bf 1}_{[0,t]}(s)h(s)x\tilde{N}(dsdx)$ and the choice done for $\gamma$ gives
$\Gamma[\int_0^th(s)dY_s]=\int_0^th^2(s)d[Y,Y]_s
$ for $h\in L^2_{loc}(dt)$.
In order to study the regularity of the random variable $ V=\int_0^t\varphi(Y_{s-})dY_s$ where $\varphi$ is Lipschitz and $\mathcal{C}^1$, we have two ways:

a) We may  represent the gradient $\sharp$ as $Y_t^\sharp=B_{[Y,Y]_t}$ where $B$ is a standard auxiliary independent Brownian motion. Then by the chain rule
$$
V^\sharp=\int_0^t\varphi^\prime(Y_{s-})(Y_{s-})^\sharp dY_s+\int_0^t\varphi(Y_{s-})dB_{[Y]_s}$$ now using $(Y_{s-})^\sharp=(Y_s^\sharp)_-$, a classical but rather tedious stochastic calculus yields
\begin{equation}\label{gamma-exemple}
\Gamma[V]=\hat{\mathbb{E}}[V^{\sharp 2}]=\sum_{\alpha\leq t}\Delta Y_\alpha^2(\int_{]\alpha}^t\varphi^\prime(Y_{s-})dY_s+\varphi(Y_{\alpha-}))^2.
\end{equation} where $\Delta Y_\alpha=Y_\alpha-Y_{\alpha-}$.
Since $V$ has real values the {\it energy image density property} holds for $V$, and $V$ has a density as soon as $\Gamma[V]$ is strictly positive a.s. what may be discussed using the relation (\ref{gamma-exemple}).

b) An other more direct way consists in applying the theorem. For this we define $\flat$ by choosing $\xi$ such that $\int_0^1\xi(r)dr=0$ and $\int_0^1\xi^2(r)dr=1$ and putting
$f^\flat=xf^\prime(x)\xi(r).
$

$1^o$. First step. We add a particle $(\alpha,x)$ i.e. a jump to $Y$ at time $\alpha$ with size $x$ what gives

\noindent$\varepsilon^+V-V=\varphi(Y_{\alpha-})x+\int_{]\alpha}^t(\varphi(Y_{s-}+x)-\varphi(Y_{s-}))dY_s
$

$2^o$. $V^\flat=0$ since $V$ does not depend on $x$, and

\noindent$(\varepsilon^+V)^\flat=(\varphi(Y_{\alpha-})x+\int_{]\alpha}^t\varphi^\prime(Y_{s-}+x)xdY_s)\xi(r)\quad$
\noindent because $x^\flat=x\xi(r)$.

$3^o$. We compute
$\gamma[\varepsilon^+V]=\int(\varepsilon^+V)^{\flat2}dr=(\varphi(Y_{\alpha-})x+\int_{]\alpha}^t\varphi^\prime(Y_{s-}+x)xdY_s)^2$

$4^o$. We take back the particle we gave, in order to compute $\int\anni\gamma[\varepsilon^+V]dN$. That gives
$$\int\anni\gamma[\varepsilon^+V]dN=\int\left(\varphi(Y_{\alpha-})+\int_{]\alpha}^t\varphi^\prime(Y_{s-})dY_s\right)^2x^2\;N(d\alpha dx)$$ and (\ref{gamma-exemple}).

We remark that both operators $F\mapsto \varepsilon^+F$, $F\mapsto(\varepsilon^+F)^\flat$ are non-local, but instead $F\mapsto \int \anni(\varepsilon^+F)^\flat\,d(N\odot\rho)$ and $F\mapsto \int\anni\gamma[\varepsilon^+F]\,dN$ are local : taking back the lent particle gives the locality. We will deepen this example in dimension $p$ in Part 5.

\section{(EID) property on the upper space from (EID) property on the
bottom space and the domain $\bbD_{loc}$}

From now on, we make additional hypotheses on the bottom structure
$(X,\cX,\nu ,\bbd,\gamma)$ which are stronger but  satisfied in
most of the
examples.\\
{\underline{Hypothesis (H1)}}: $X$ admits a partition of the form:
$X=B\bigcup (\bigcup_{k=1}^{+\infty}A_k )$ where for all $k$,
$A_k\in\cX$ with $\nu (A_k )<+\infty$ and $\nu (B)=0$, in such a way that for any $k\in\Ne$ may be defined a local  Dirichlet structure with carr\'e du champ:
\[ \cS_k =(A_k ,\cX_{|A_k}, \nu_{|A_k }, \bbd_k,\gamma_k),\]
with
\[\forall f\in\bbd, \, f_{|A_k }\in\bbd_k \makebox{ and } \gamma
[f]_{|A_k}=\gamma_k [f_{|A_k}].\]
 {\underline{Hypothesis (H2)}}: Any finite product of structures
 $\cS_k$ satisfies (EID).
 \begin{Rq} In many examples where $X$ is a topological space, (H1)
 is satisfied by choosing for $(A_k )$ $k\in\Ne$ a regular
 open set. \\
 Let us remark that (H2) is satisfied for the structures studied in
Part 2.
 \end{Rq}
 The main result of this section is the following:
 \begin{Pro}\label{proEID} If the bottom structure $(X,\cX,\nu
,\bbd,\gamma)$  satisfies (H1) and (H2) then the upper structure
$(\Omega ,\cA ,\bbP ,\bbD,\Gamma )$ satisfies (EID).
\end{Pro}\begin{proof}
For all $k\in\Ne$, since $\nu (A_k )<+\infty$, we consider an upper
structure $S_k =(\Omega_k, \cA_k ,$ $\bbP_k ,\bbD_k ,\Gamma_k )$
associated to $\cS_k$ as a direct application of the construction by
product (see \S 3.3.2 above or Bouleau \cite{bouleau3} Chap. VI.3).\\
Let $k\in\Ne$, we denote  by $N_k$ the corresponding random Poisson
measure on $A_k$ with intensity $\nu_{|A_k}$ and we consider
$N^{\ast}$ the random Poisson measure on $X$ with intensity $\nu$,
defined on the product probability space
\[ (\Omega^{\ast} , \cA^{\ast} ,\bbP^{\ast} )=\prod_{k=1}^{+\infty}
(\Omega_k  ,\cA_k ,\bbP_k ),\] by
\[N^{\ast}=\sum_{k=1}^{+\infty} N_k .\]
In a natural way, we consider the product Dirichlet structure
\[ S^{\ast} =(\Omega^{\ast} , \cA^{\ast} ,\bbP^{\ast},\bbD^{\ast} ,\Gamma^{\ast}
)=\prod_{k=1}^{+\infty} S_k .\] In the third Part, we have built
using the Friedrichs argument, the Dirichlet structure
\[ S=(\Omega ,\cA ,\bbP ,\bbD ,\Gamma ),\]
let us now make the link between those structures.\\
First of all, thanks to Theorem 2.2.1 and Proposition 2.2.2. of
Chap. V in Bouleau-Hirsch  \cite{bouleau-hirsch2},  we know that a function
$\varphi$ in $L^2 (\bbP^{\ast} )$ belongs  to $\bbD^{\ast}$ if and
only if
\begin{enumerate}
\item For all $k\in\Ne$ and $\prod_{n\neq k}\bbP_n$-almost all $\xi_1
,\cdots ,\xi_{k_1},\xi_{k+1},\cdots $, the map
\[ \xi\mapsto \varphi (\xi_1
,\cdots ,\xi_{k_1},\xi ,\xi_{k+1},\cdots )\] belongs to $\bbD_k$.
\item $\sum_k \Gamma_k [\varphi ]\in L^1 (\bbP^{\ast})$ and we have
$\Gamma^{\ast}[\varphi ]=\sum_k \gamma_k [\varphi ]$.
\end{enumerate}
Consider $f\in\bbd\cap L^1 (\gamma )$ then clearly
\[N(f)=\sum_k N_k (f_{|A_k })\]
belongs to $\bbD^{\ast}$ and in the same way
\[ e^{i\tN (f)}=\prod_k e^{i\tN_k (f_{|A_k})} \in \bbD^{\ast}.\]
Moreover, by hypothesis (H1):
\begin{eqnarray*}
\Gamma^{\ast} [e^{i\tN (f)}]&=&\sum_k | \prod_{l\neq
k}e^{i\tN_l (f_{|A_l})}|^2 \Gamma_k[e^{i\tN_k (f_{|A_k})}]=\sum_k N_k (\gamma[f]_{|A_k})\\
&=& N (\gamma[f])=\Gamma [e^{i\tN (f)}].
\end{eqnarray*}
Thus as $\cD_0$ is dense in $\bbD$, we conclude that $\bbD \subset
\bbD^{\ast}$ and $\Gamma =\Gamma^{\ast}$ on $\bbD$.\\
As for all $k$, $S_k$ is a product structure, thanks to hypothesis
(H2) and Proposition 2.2.3 in Bouleau-Hirsch  \cite{bouleau-hirsch2} Chapter V, we conclude that $S^{\ast}$  satisfies (EID) hence $S$ too.
\end{proof}
\noindent{\bf Main case.} Let $N$ be a random Poisson measure on $\mathbb{R}^d$ with intensity measure $\nu$ satisfying one of the following conditions :

i) $\nu=k\,dx$ and a function $\xi$ (the carr\'e du champ coefficient matrix) may be chosen such that hypotheses (HG) hold (cf \S2.1)

ii) $\nu$ is the image by a Lipschitz injective map of a measure satisfying (HG) on $\mathbb{R}^q$, $q\leq d$,

iii) $\nu$ is a product of measures like ii),

\noindent then the associated Dirichlet structure $(\Omega,\mathcal{A},\mathbb{P},\mathbb{D},\Gamma)$ constructed (cf \S3.2.4) with $\nu$ and the carr\'e du champ obtained by the $\xi$ of i) or induced by operations ii) or iii) satisfies (EID).\\

We end this section by a few remarks on the localization of this
structure which permits to extend the functional calculus related to
$\Gamma$ or $\sharp$ to bigger spaces than $\bbD$, which is
often convenient from a practical point of view.\\
Following Bouleau-Hisrch (see \cite{bouleau-hirsch2}} p. 44-45) we recall that
$\bbD_{loc}$ denotes the set of functions $F:\Omega \rightarrow\R$
such that there exists a sequence $(E_n )_{n\in\Ne}$ in $\cA$ such
that \[ \Omega =\bigcup_n E_n \makebox{ and } \forall n\in\Ne ,\
\exists F_n \in\bbD\ F_n =F \makebox{ on } E_n .\] Moreover if
$F\in\bbD_{loc}$, $\Gamma [F]$ is well-defined and satisfies
$(EID)$ in the sense that
\[ F_\ast (\Gamma [F]\cdot P)\ll \lambda^1.\]
More generally, if $(\Om, \cA ,\bbP ,\bbD ,\Gamma )$ satisfies
(EID),
\[ \forall F\in (\bbD_{loc})^n ,\ F_\ast (det\Gamma[F]\cdot\bbP )\ll
\lambda^n .\] We can consider another space bigger than
$\bbD_{loc}$ by considering a partition of $\Omega$ consisting in
a sequence of sets with negligible boundary. More precisely, we
denote by $\bbD_{LOC}$ the set of functions $F:\Omega
\rightarrow\R$ such that there exists a sequence of disjoint sets
$(A_n )_{n\in\Ne}$ in $\cA$  such that  $\bbP (\Omega
\setminus\bigcup_n A_n ) =0$ and \[ \forall n\in\Ne ,\ \exists F_n
\in\bbD\ F_n =F \makebox{ on } A_n .\] One can easily verify that
it contains the localized domain of any structure $S^{\ast}$ as
considered in the proof of Proposition \ref{proEID}, that $\Gamma$
is well-defined on $\bbD_{LOC}$, that the functional calculus
related  to $\Gamma$ or $\sharp$ remains valid and that it
satisfies (EID) i.e. if $(\Om, \cA ,\bbP ,\bbD ,\Gamma )$
satisfies (EID),
\[ \forall F\in (\bbD_{LOC})^n ,\ F_\ast (det\Gamma[F]\cdot\bbP )\ll
\lambda^n .\]
\section{Examples}
\subsection{Upper bound of a process on [0,t]}
Let $Y$ be a real process with stationary independent increments satisfying the hypotheses of example 3.3.2. 

We consider a real c\`adl\`ag process $K$ independent of $Y$ and put $H_s=Y_s+K_s$.
\begin{Pro}\label{proEx1} If $\sigma(\mathbb{R}\backslash\{0\})=+\infty$ and if $\mathbb{P}[\sup_{s\leq t}H_s=H_0]=0$,  the random variable $\sup_{s\leq t}H_s$ possesses a density.
\end{Pro}
\begin{proof} a) We may suppose that $K$ satisfies $\sup_{s\leq t}|K_s|\in L^2$. Indeed, if random variables $X_n$ have densities and $\mathbb{P}[X_n\neq X]\rightarrow 0$, then $X$ has a density. Hence the assertion is obtained by considering $(K_s\wedge k)\vee(-k)$.

b) Let us put $M=\sup_{s\leq t}H_s$. Applying the lent particle method gives
$$\begin{array}{rl}
(\varepsilon^+M)(\alpha,x)&=\sup_{s\leq t}((Y_s+K_s)1_{\{s<\alpha\}}+(Y_s+x+K_s)1_{\{s\geq \alpha\}})\\
&=\max(\sup_{s<\alpha}(Y_s+K_s),\sup_{s\geq\alpha}(Y_s+x+K_s))\\
\gamma[\varepsilon^+M](\alpha,x)&={\bf 1}_{\{\sup_{s\geq\alpha}(Y_s+x+K_s)\geq\sup_{s<\alpha}(Y_s+K_s)\}}\gamma[j](x)
\end{array}
$$
where $j$ is the identity map $j(x)=x$.

We take back the lent particle before integrating with respect to $N$ and obtain, since $\gamma[j](x)=x^2$,
$$\Gamma[M]=\int\anni\gamma[\varepsilon^+M]\,N(d\alpha dx)=
\sum_{\alpha\leq t}\Delta Y_\alpha^2{\bf 1}_{\{\sup_{s\geq\alpha}(Y_s+K_s)\geq\sup_{s<\alpha}(Y_s+K_s)\}}.$$
As $\sigma(\mathbb{R}\backslash\{0\})=+\infty$, Y has infinitely many jumps on every time interval, so that
$$\Gamma[M]=0\quad\Rightarrow\quad\forall\alpha\leq t\; \sup_{s\geq\alpha}(Y_s+K_s)<\sup_{s<\alpha}(Y_s+K_s)$$
and choosing $\alpha$ decreasing to zero, we obtain
$$\Gamma[M]=0\quad\Rightarrow\quad \sup_{t\geq s\geq 0}H_s=H_0$$ and the proposition.
\end{proof}
It follows that any real L\'evy process $X$ starting at zero and immediately entering $\mathbb{R}_+^\ast$, whose L\'evy measure dominates a measure $\sigma$ satisfying Hamza's condition and infinite, is such that $\sup_{s\leq t}X_s$ has a density.

\subsection{Regularizing properties of L\'evy processes}
Let $Y$ be again a real process with stationary independent increments satisfying the hypotheses of example 3.3.2. By Hamza's condition, hypothesis (H1) is fulfilled and hypothesis (H2) ensues from Theorem 2, so that the upper structure verifies (EID).

Let $S$ be an $\mathbb{R}^p$-valued semi-martingale independent of $Y$. We will say that $S$ is pathwise $p$-dimensional on $[0,t]$ if almost every sample path of $S$ on $[0,t]$ spans a $p$-dimensional vector space.

We consider the $\mathbb{R}^p$-valued process $Z$ whose components are given by
$$Z_t^1=S_t^1+Y_t^1\quad\mbox{and}\quad Z_t^i=S_t^i\quad\forall i\geq 2$$ and the stochastic integral
$$R=\int_0^t\psi(Z_{s-})\,dZ_s$$
where $\psi$ is a Lipschitz and $\mathcal{C}^1$ mapping from $\mathbb{R}^p$ into $\mathbb{R}^{p\times p}$.
\begin{Pro} If  $\sigma(\mathbb{R}\backslash\{0\})=+\infty$, if the Jacobian determinant of the column vector $\psi_{.1}$ does not vanish and if $R$ is pathwise $p$-dimensional on $[0,t]$, then the law of $R$ is absolutely continuous with respect to $\lambda^p$.
\end{Pro}
\begin{proof} We apply the lent particle method. Putting $\overline{x}=(x,0,\ldots,0)$  and \break $R^i=\sum_j\int_0^t\psi_{ij}(Z_{s-})\,dZ_s^j$, we have
$$\varepsilon^+R^i-R^i=\psi_{i1}(Z_{\alpha-})x+\int_{]\alpha}^t(\psi_{i1}(Z_{s-}+\overline{x})-\psi_{i1}(Z_{s-}))\,dY_s$$
as in example 3.3.2,
$$(\varepsilon^+R^i)^\flat=(\psi_{i1}(Z_{\alpha-})x+\int_{[\alpha}^t\partial_1\psi_{i1}(Z_{s-}+\overline{x})x\,dY_s)\xi(r)$$
and\\
$\gamma[\varepsilon^+R^i,\varepsilon^+R^j]=$
$$\left(\psi_{i1}(Z_{\alpha-})+\int_{[\alpha}^t\partial_1\psi_{i1}(Z_{s-}+\overline{x})\,dY_s\right)\left(\psi_{j1}(Z_{\alpha-})+\int_{[\alpha}^t\partial_1\psi_{j1}(Z_{s-}+\overline{x})\,dY_s\right)x^2.$$
We take back the lent particle before integrating in $N$ :
$$\Gamma[R^i,R^j]=\int\anni(\gamma[\varepsilon^+R^i,\varepsilon^+R^j])\,dN= \sum_{\alpha\leq t}\Delta Y_\alpha^2 U_\alpha U_\alpha^t$$
where $U_\alpha$ is the column vector $\psi_{.1}(Z_{\alpha-})+\int_{[\alpha}^t\partial_1\psi_{.1}(Z_{s-})\,dY_s$.

Let $JT$ be the set of jump times of $Y$ on $[0,t]$, we conclude that
$$\det \Gamma[R,R^t]=0\quad\Leftrightarrow\quad \dim\mathcal{L}(U_\alpha\,;\,\alpha\in JT)<p.$$
Let $A=\{\omega\,:\,\dim\mathcal{L}(U_\alpha\,;\,\alpha\in
JT)<p\}$. Reasoning  on $A$, there exist
$\lambda_1,\ldots,\lambda_p$ such that
\begin{equation}\label{ex2somme}
\sum_{k=1}^p\lambda_k\left(\psi_{k1}(Z_{\alpha-})+\int_{[\alpha}^t\partial_1\psi_{k1}(Z_{s-})\,dY_s\right)=0\quad\forall\alpha\in JT,
\end{equation}
now, since $\sigma(\mathbb{R}_+\backslash\{0\})=+\infty$, $JT$ is a dense countable subset
of $[0,t]$, so that taking left limits in (\ref{ex2somme}), using
(\ref{ex2somme}) anew and the fact that $\psi$ is $\mathcal{C}^1$,
we obtain
$$\sum_{k=1}^p\lambda_k\psi_{k1}(Z_{\alpha-})=0\quad\forall\alpha\in JT\quad\mbox{hence}\quad\forall\alpha\in]0,t]$$ thus, on $A$, we have $\dim \mathcal{L}(\psi_{.1}(Z_{s-});s\in]0,t])<p$.

Then EID property yields the conclusion.
\end{proof}

The lent particle method and (EID) property may be applied to density results for solutions of stochastic differential equations driven by L\'evy processes or random measures under Lipschitz hypotheses. Let us mention also that the gradient $\sharp$ defined in \S 3.2 has the property to be easily iterated, this allows to obtain conclusions on $C^\infty$-regularity in the case of smooth coefficients. These applications will be investigated in forthcoming articles.
\subsection{A regular case violating H\"ormander conditions} In spite of the difficulty of the proofs, applying the method is quite easy.   This will be pushed forward in another article, we are just showing here an extremely simple case, example of situations  rarely taken in account in the literature.

\noindent a) Let us consider the following sde driven by a two dimensional Brownian motion
\begin{equation}\label{diffusion}\left\{\begin{array}{rl}
X^1_t&=z_1+\int_0^tdB^1_s\\
X^2_t&=z_2+\int_0^t2X^1_sdB^1_s+\int_0^tdB^2_s\\
X^3_t&=z_3+\int_0^tX^1_sdB^1_s+2\int_0^tdB^2_s.
\end{array}\right.
\end{equation}
This diffusion is degenerate and the H\"ormander conditions are not fulfilled. The generator is $A=\frac{1}{2}(U_1^2+U_2^2)+V$
and its adjoint $A^\ast=\frac{1}{2}(U_1^2+U_2^2)-V$ with  
$U_1=\frac{\partial}{\partial x_1}+2x_1\frac{\partial}{\partial x_2}+x_1\frac{\partial}{\partial x_3}$,
$U_2=\frac{\partial}{\partial x_2}+2\frac{\partial}{\partial x_3}$ and $V=-\frac{\partial}{\partial z_2}-\frac{1}{2}\frac{\partial}{\partial z_3}$. The Lie brackets of these vectors vanish and the Lie algebra is of dimension 2 :
the diffusion remains on the quadric of equation
$\frac{3}{4}x_1^2-x_2+\frac{1}{2}x_3-\frac{3}{4}t=C.$

\noindent b) Let us now consider the same equation driven by a L\'evy process :

$$\left\{\begin{array}{rl}
Z^1_t&=z_1+\int_0^tdY^1_s\\
Z^2_t&=z_2+\int_0^t2Z^1_{s_-}dY^1_s+\int_0^tdY^2_s\\
Z^3_t&=z_3+\int_0^tZ^1_{s_-}dY^1_s+2\int_0^tdY^2_s
\end{array}\right.
$$
under hypotheses on the L\'evy measure  such that the bottom space may be equipped with the carr\'e du champ operator 
$\gamma[f]=y_1^2f^{\prime 2}_1+y_2^2f^{\prime 2}_2$ satisfying (BC) and our hypotheses yielding EID.
Let us apply the lent particle method.
$$\mbox{ For }\alpha\leq t\qquad\varepsilon^+_{(\alpha,y_1,y_2)}Z_t=Z_t+\left(
\begin{array}{c}
y_1\\
2Y^1_{\alpha-}y_1+2\int_{]\alpha}^ty_1dY^1_s+y_2\\
Y^1_{\alpha-}y_1+\int_{]\alpha}^ty_1dY^1_s+2y_2
\end{array}\right)=Z_t+\left(
\begin{array}{c}
y_1\\
2y_1Y_t^1+y_2\\
y_1Y_t^1+2y_2
\end{array}\right).
$$
where we have used $Y^1_{\alpha-}=Y^1_\alpha$ because $\varepsilon^+$ send into $\mathbb{P}\times \nu$ classes. That gives
$$\gamma[\varepsilon^+Z_t]=\left(
\begin{array}{lcr}
y_1^2&y_1^22Y^1_t&y_1^2Y^1_t\\
id&y_1^24(Y^1_t)^2+y_2^2&y_1^22(Y^1_t)^2+2y_2^2\\
id&id&y_1^2(Y_t^1)^2+4y_2^2
\end{array}
\right)$$
and 
$$\varepsilon^-\gamma[\varepsilon^+Z_t]=\left(
\begin{array}{lcr}
y_1^2&y_1^22(Y^1_t-\Delta Y_\alpha^1)&y_1^2(Y^1_t-\Delta Y_\alpha^1)\\
id&y_1^24(Y^1_t-\Delta Y_\alpha^1)^2+y_2^2&y_1^22(Y^1_t-\Delta Y_\alpha^1)^2+2y_2^2\\
id&id&y_1^2(Y_t^1-\Delta Y_\alpha^1)^2+4y_2^2
\end{array}
\right)$$
hence
$$\Gamma[Z_t]=\sum_{\alpha\leq t}(\Delta Y_\alpha^1)^2
\left(
\begin{array}{lcr}
1&2(Y^1_t-\Delta Y_\alpha^1)&(Y^1_t-\Delta Y_\alpha^1)\\
id&4(Y^1_t-\Delta Y_\alpha^1)^2&2(Y^1_t-\Delta Y_\alpha^1)^2\\
id&id&(Y_t^1-\Delta Y_\alpha^1)^2
\end{array}
\right)+(\Delta Y^2_\alpha)^2
\left(
\begin{array}{lcr}
0&0&0\\
0&1&2\\
0&2&4
\end{array}
\right).
$$
If the L\'evy measures of  $Y^1$ and $Y^2$ are infinite, it follows that $Z_t$ has a density as soon as 
$$
\mbox{dim }\mathcal{L}\left\{\left(
\begin{array}{c}
1\\
2(Y^1_t-\Delta Y_\alpha^1)\\
(Y^1_t-\Delta Y_\alpha^1)
\end{array}\right), \left(\begin{array}{c}
0\\
1\\
2
\end{array}\right)\quad\alpha\in JT\right\}=3.$$
But  $Y^1$ possesses necessarily jumps of different sizes, hence  $Z_t$ has a density on $\mathbb{R}^3$.

It follows that the integro-differential operator
$$\tilde{A}f(z)=
\int\!\left[f(z)-f\!
\left(\begin{array}{c}
z_1+y_1\\
z_2+2z_1y_1+y_2\\
z_3+z_1y_1+2y_2
\end{array}\right)
-(f^\prime_1(z)\;f^\prime_2(z)\;f^\prime_3(z))\left(
\begin{array}{c}
y_1\\
2z_1y_1+y_2\\
z_1y_1+2y_2
\end{array}\right)
\right]\sigma(dy_1dy_2)
$$
is hypoelliptic at order zero, in the sense that its semigroup $P_t$ has a density. No minoration is supposed of the growth of the L\'evy measure near 0 as assumed by many authors. 

This result implies that for any L\'evy process $Y$ satisfying the above hypotheses, even a subordinated one in the sense of Bochner, the process $Z$ is never subordinated of the Markov process $X$ solution of equation (\ref{diffusion}).

 Ecole des Ponts,\\
ParisTech, Paris-Est\\
6 Avenue Blaise Pascal\\ 77455 Marne-La-Vallée Cedex 2
FRANCE\\bouleau@enpc.fr \\  \\ Equipe Analyse et Probabilités,
\\Universit\'{e} d'Evry-Val-d'Essonne,\\Boulevard François Mitterrand\\
91025 EVRY Cedex FRANCE\\ldenis@univ-evry.fr

\begin{thebibliography}{00}
\bibitem{akr}{\sc Albeverio S., Kondratiev Y.} and {\sc R\"ockner M.} "Differential geometry of Poisson space" {\it C. R. Acad. Sci. Paris}, t. 323, sI, 1129-1134, (1996); and "Analysis and geometry on configuration spaces" {\it J. Funct. Analysis} 154, 444-500, (1998).
\bibitem{albeverio-rockner}{\sc Albeverio S.} and {\sc R\"ockner M.} "Classical Dirichlet forms on topological vector spaces --- closability and a Cameron-Martin formula" {\it J. Funct. Analysis}, 88, 395-436,  (1990).
\bibitem{applebaum}{\sc Applebaum D.} "Universal Malliavin calculus in Fock space and L\'evy-It\^o spaces" preprint (2008).
\bibitem{bichteler-gravereaux-jacod}{\sc Bichteler K., Gravereaux J.-B., Jacod J.} {\it Malliavin Calculus for Processes with Jumps} (1987)
\bibitem{bouleau1}{\sc Bouleau N.} "D\'ecomposition de l'\'energie par niveau de potentiel" {\it Lect. Notes in M. 1096}, Springer(1984).
\bibitem{bouleau2}{\sc Bouleau N.} "Construction of Dirichlet structures" {\it in : Potential Theory-ICPT94}, Kr\'al, Luke\u{s}, Netuka, Vesel\'y, eds, De Gruyter (1996).
\bibitem{bouleau3}{\sc Bouleau N.} {\it Error Calculus for Finance and Physics, the Language of Dirichlet Forms}, De Gruyter (2003).
\bibitem{bouleau4}{\sc Bouleau N.} "Error calculus and regularity of Poisson functionals: the lent particle method" C. R. Acad. Sc. Paris,  Math\'ematiques, 
Vol 346, n13-14, (2008), p779-782.
\bibitem{bouleau-hirsch1}{\sc Bouleau N.} and {\sc Hirsch F.}"Formes de Dirichlet g\'en\'erales et densit\'e des variables al\'eatoires r\'eelles sur l'espace de Wiener" {\it J. Funct. Analysis} 69, 2, 229-259,  (1986).
\bibitem{bouleau-hirsch2}{\sc Bouleau N.} and {\sc Hirsch F.} {\it Dirichlet Forms and Analysis on Wiener Space} De Gruyter (1991).
\bibitem{comtet}{\sc Comtet L.} {\it Advanced Combinatorics}, Springer (1974).
\bibitem{coquio}{\sc Coquio A.} "Formes de Dirichlet sur l'espace canonique de Poisson et application aux \'equations diff\'erentielles stochastiques" {\it Ann. Inst. Henri Poincar\'e} vol 19, n1, 1-36, (1993)
\bibitem{dellacherie}{\sc Dellacherie C., Maisonneuve B.} and {\sc Meyer P.-A.} {\it Probabilit\'es et Potentiel} Chap XVII \`a XXIV, Hermann 1992.
\bibitem{denis}{\sc Denis L.} "A criterion of density for
solutions of Poisson-driven SDEs" {\it Probab. Theory Relat.
Fields} 118, 406-426,  (2000).
\bibitem{federer}{\sc Federer H.} {\it Geometric Measure Theory},
Springer (1969).
\bibitem{fukushima-oshima-takeda}{\sc Fukushima M., Oshima Y.} and {\sc Takeda M.} {\it Dirichlet Forms and Symmetric Markov Processes} De Gruyter (1994).
\bibitem{ishikawa-kunita}{\sc Ishikawa Y.} and {\sc Kunita H.} "Malliavin calculus on the Wiener-Poisson space and its application to canonical SDE with jumps" {\it Stoch. Processes and their App.} 116, 1743-1769, (2006).
\bibitem{ito}{\sc Ito K.} "Spectral type of the shift transformation of differential processes with stationary increments" {\it Trans. Amer. Math. Soc.} 81, 253-263, (1956)
\bibitem{klr}{\sc Kondratiev Y., Lytvynov E.} and {\sc R\"ockner M.} "The semigroup of the Glauber dynamics of a continuous system of free particles" arXiv: math/0407359v1 (2004).
\bibitem{martin-lof}{\sc Martin-L\"of A.} "Limit theorems for the motion of a Poisson system of independent Markovian particles with high density" {\it Z. Wahrscheinlichkeitstheorie verw. Gebiete} 34, 205-223, (1976).
\bibitem{meyer}{\sc Meyer P.-A.} "El\'ements de probabilit\'es quantiques" {\it S\'em. Prob. XX}, Lect. Notes in M. 1204, Springer (1986).
\bibitem{ma-rockner1}{\sc Ma Z.-M.} and {\sc R\"ockner M.} {\it Introduction to the Theory of (non-symmetric) Dirichlet Forms}, Springer 1992).
\bibitem{ma-rockner2}{\sc Ma} and {\sc R\"ockner M.} "Construction of diffusion on configuration spaces" {\it Osaka J. Math.} 37, 273-314,  (2000).

\bibitem{neveu}{\sc Neveu J.} {\it Processus Ponctuels}, Lect. Notes in M. 598, Springer (1977).
\bibitem{nualart-vives}{\sc Nualart D.} and {\sc Vives J.} "Anticipative calculus for the Poisson process based on the Fock space", {\it S\'em. Prob. XXIV}, Lect. Notes in M. 1426, Springer (1990).
\bibitem{picard}{\sc Picard J.}"On the existence of smooth densities for jump processes" {\it Probab. Theorie Relat. Fields} 105, 481-511, (1996)
\bibitem{picard2}{\sc Picard J.} "Brownian excursions, stochastic integrals and representation of Wiener functionals" {\it Elec. Jour. Probability} 11, 199-248, (2006).
\bibitem{privault} {\sc Privault N.} "A pointwise equivalence of gradients on configuration spaces", {\it C. Rendus Acad. Sc. Paris}, 327, 7, 677-682, (1998).
\bibitem{rockner-wielens}{\sc R\"ockner M.} and {\sc Wielens N.} "Dirichlet forms --- Closability and change of speed measure", {\it Infinite Dimensional Analysis and Stochastic Processes}, Research Notes in M., Albeverio ed. Pitman 124, 119-144, (1985).
\bibitem{rota-wallstrom}{\sc Rota G.-C.} and {\sc Wallstrom T.} "Stochastic integrals: a combinatorial approach" {\it Ann. of Probability} 25, 3, 1257-1283, (1997).
\bibitem{song}{\sc Song Sh.} "Admissible vectors and their associated Dirichlet forms" {\it Potential Analysis} 1, 4, 319-336, (1992).
\bibitem{surgailis}{\sc Surgailis D.} "On multiple Poisson stochastic integrals and associated Markov processes" {\it Probability and Mathematical Statistics} 3, 2, 217-239, (1984)
\bibitem{wu}{\sc Wu L.} "Construction de l'op\'erateur de Malliavin sur l'espace de Poisson" {\it S\'em. Probabilit\'e XXI} Lect. Notes in M. 1247, Springer (1987).
\end{thebibliography}
 \end{document}